\documentclass{amsart}
\usepackage{graphicx}
\usepackage{epsfig}
\usepackage{graphicx}
\usepackage{color}
\usepackage{amstext}
\usepackage{setspace}
\usepackage{amsmath}
\usepackage{dsfont}
\usepackage{hyperref}
\usepackage[notref,notcite]{}
\usepackage{subfigure}

\newtheorem{thm}{Theorem}[section]

\newtheorem{lem}[thm]{Lemma}
\newtheorem{prop}[thm]{Proposition}
\theoremstyle{definition}

\theoremstyle{remark}

\numberwithin{equation}{section}

\newcommand{\abs}[1]{\lvert#1\rvert}

\newcommand{\e}{\varepsilon}

  \def\cD{\mathcal{D}}
\def\cF{\mathcal{F}}

\def\cN{\mathcal{N}}

\def\bE{\mathbb{E}}
\def\bN{\mathbb{N}}
\def\bP{\mathbb{P}}

\def\bR{\mathbb{R}}
\def\bZ{\mathbb{Z}}

\def\P{\bP} \def\E{\bE}

\def\mE{\mathbf{E}}

\def\mP{\mathbf{P}}
\def\m1{\mathbf{1}}

\newcommand{\be}{\begin{equation}}
\newcommand{\ee}{\end{equation}}
\newcommand{\nn}{\nonumber}
\newcommand{\fl}[1]{\lfloor{#1}\rfloor}

\DeclareMathOperator{\Var}{Var}   
\DeclareMathOperator{\esssup}{ess\,sup}
\DeclareMathOperator{\essinf}{ess\,inf}
\def\wt{\widetilde}   
\def\ppa{p} 
\def\ppb{s} 

\allowdisplaybreaks[1]

\begin{document}



\title[Last-passage shape]{Properties of the limit shape for some last passage    growth models in   random
environments}
\author{Hao Lin}
\address{Hao Lin, Department of Mathematics, University of
                Wisconsin-Madison, Madison, WI 53706, USA}
\email{lin@math.wisc.edu}

\author{Timo Sepp\"{a}l\"{a}inen}
\address{Timo Sepp\"{a}l\"{a}inen, Department of Mathematics, University of
                Wisconsin-Madison, Madison, WI 53706, USA}
\email{seppalai@math.wisc.edu}
\urladdr{http://www.math.wisc.edu/~seppalai}
\thanks{T.\ Sepp\"al\"ainen was partially supported by
National Science Foundation grants DMS-0701091  and DMS-1003651,
 and by the Wisconsin Alumni Research
Foundation.}
\date{\today}
\begin{abstract}
We study directed   last passage percolation on the
planar square  lattice whose weights
have general distributions, or equivalently,    queues in series with
general service distributions.
Each row of the last passage
  model has its own randomly chosen  weight distribution.
    We investigate the limiting time
constant close to the boundary of the quadrant.  Close to the $y$-axis,
where  the number of
random distributions averaged over stays large,  the limiting time constant
takes the same  universal form as in the homogeneous model.
But close to the $x$-axis we see the effect of the tail of the distribution of
the random environment.
\end{abstract}
\maketitle



\section{Introduction}
This paper studies the limit shapes of some last passage percolation models in
random environments, specifically, the corner growth model and
two Bernoulli models with different rules for admissible paths.

We introduce the corner growth model through its queueing interpretation.
   Consider
service stations in series, labeled $0, 1, 2, \dotsc, \ell$, each
with unbounded waiting room and first-in first-out (FIFO) service
discipline. Initially customers $0, 1, 2, \dotsc, k$ are queued up
at server $0$. At time $t=0$ customer $0$ begins service with server
$0$. Each customer moves through the system of servers  in order, joining
the queue at server $j+1$ as soon as service with server $j$ is complete.
After
customer $i$ departs server $j$, server $j$ starts serving customer $i+1$ immediately
if $i+1$ has been waiting in the queue,  or then waits for
customer $i+1$ to arrive from   station $j-1$.
Customers stay ordered throughout the process.  Let $X(i,j)$ be
the service time that customer $i$ needs at station $j$, and
$T(k,\ell)$  the time when customer $k$ completes service with
server $\ell$.

Asymptotics for $T(k,\ell)$  as $k$ and $\ell$ get large   have been
investigated a great deal   in the past two decades.
A seminal paper  by
  Glynn-Whitt \cite{glyn-whit}   studied the case of
i.i.d.\ $\{X(i,j)\}$.   They took advantage
of the   connection with
directed  last-passage percolation given by  the identity
\be T(k,\ell)=\max_\pi \sum_{(i,j)\in\pi} X(i,j).  \label{last-pass-1}\ee
The maximum is taken over
non-decreasing nearest-neighbor lattice paths $\pi\subseteq\bZ_+^2$
 from $(0,0)$ to $(k,\ell)$
that are  of  the form
$\pi=\{(0,0)=(x_0,y_0), (x_1,y_1),\dotsc, (x_{k+\ell},y_{k+\ell})=(k,\ell)\}$
where $(x_i,y_i)-(x_{i-1},y_{i-1})=(1,0)$ or $(0,1)$.
A quick inductive proof  of \eqref{last-pass-1}  together  with earlier references to this
observation  can be found
in  \cite{glyn-whit} (see Prop.~2.1).
   This particular  last passage model is also known as the
{\sl corner growth model}.

 Next we add   a random environment to  this queueing  model.  The environment is a
sequence $\{F_j: j\in\bZ_+\}$ of probability distributions, generated by a
probability measure-valued ergodic or i.i.d.\ process with distribution $\bP$.
Given the sequence $\{F_j\}$, we assume that the
 variables  $\{X(i,j)\}$  are independent and  $X(i,j)$ has distribution $F_j$.
In the queueing picture this means  that  the service times $\{X(i,j): i\in\bZ_+\}$ at service
station  $j$ have  common    distribution $F_j$,  and at the outset
 the distributions $\{F_j: j\in\bZ_+\}$ themselves are chosen randomly according
 to some given law $\bP$.    Obviously  the  labels ``customer'' and ``server''
 are interchangeable because we can switch around the roles of the indices $i$
 and $j$.

 The asymptotic regime we consider for $T(k,\ell)$ is the {\sl hydrodynamic}  one where
 $k$ and $\ell$ are both of order $n$ and $n$ is taken to $\infty$.  Under some moment  assumptions
 standard subadditive considerations and approximations
 imply the existence  of the deterministic limit
 \[\Psi(x,y)=\lim_{n\to\infty} n^{-1} T(\fl{nx},\fl{ny})\qquad\text{ for all $(x,y)\in\bR_+^2$.}\]
 Only in the case where the distributions $F_j$ are exponential or
 geometric has it been possible to describe explicitly the limit $\Psi$.
 This is the case of $\cdot\,/M/1$ queues in series,  which in terms of
 interacting particle systems is the same as studying either the
 totally asymmetric simple exclusion process or the zero-range process
 with constant jump rate.  For rate 1 i.i.d.\ exponential $\{X(i,j)\}$
the limit $\Psi(x,y)=(\sqrt x+\sqrt y\,)^2$ was first derived by Rost \cite{Rost1981}
in a seminal paper on hydrodynamic limits of asymmetric exclusion processes.
The random environment model with exponential $F_j$'s was studied
in \cite{andj-etal, krug-ferr, TimoKrug}.

Let us now set aside the queueing motivation and consider    the
last-passage model on the first quadrant   $\bZ_+^2$ of the planar integer lattice,
defined by the nondecreasing lattice paths and
the random  weights $\{X(i,j)\}$.    For the queueing application it is natural to
assume the weights nonnegative, but in the general last-passage situation there
is no reason to restrict   to nonnegative weights.

The ideal limit shape result would have some degree of universality, that is,
apply to a broad class of distributions.  Such
  results have been obtained only close to the
boundary:  in    \cite{Martin2004}  Martin showed that in the i.i.d.\ case, under
suitable moment hypotheses and
 as  $\alpha \searrow 0$,
\be \Psi(1, \alpha) = \mu + 2\sigma \sqrt{\alpha}+ o(\sqrt{\alpha}),
\label{martin-1} \ee where $\mu$ and $\sigma^2$ are the common mean
and variance of the weights  $X(i,j)$. Also, the $o(\sqrt\alpha)$
term in the statement means that   $\lim_{\alpha \searrow 0}
\alpha^{-1/2}\bigl[ \Psi(1,\alpha) - \mu - 2\sigma \sqrt{\alpha} \,
\bigr] =0. $ In the i.i.d.\ case $\Psi$ is symmetric so the same
holds for $\Psi(\alpha, 1)$.

Our goal is to find the form Martin's result takes
in the random environment  setting.  $\Psi$ is no longer necessarily symmetric
since the distribution  of the array $\{X(i,j)\}$ is not invariant under transposition.
So we must ask the question
separately for $\Psi(1, \alpha)$ and $\Psi(\alpha, 1)$.

It turns out that for $\Psi(\alpha, 1)$, where the number of rows stays large relative
to the number of columns,
   the fluctuations  of the
  environment average out to the degree that our result in Theorem \ref{thm:main}  below is
 essentially identical to Martin's  result in the homogeneous environment.
 We  still have
$\Psi(\alpha,1) = \mu + 2\sigma \sqrt{\alpha}+ o(\sqrt{\alpha})$ as
$\alpha \searrow 0$, where now $\mu$ is the mean as before but
$\sigma^2$ is the average of the ``quenched'' variance.   That is,
if we let $\mu_0=\int x\,dF_0(x)$ and
 $\sigma^2_0=\int (x-\mu_0)^2\,dF_0(x)$ denote the mean and  variance of the random distribution
$F_0$, and $\bE$   expectation under $\bP$, then $\mu=\bE(\mu_0)$ and $\sigma^2 = \bE( \sigma^2_0)$.

The case $\Psi(1, \alpha)$ does not possess a clean result such as the one above.
Even though we are studying the deterministic  limit obtained {\sl after} $n$ has been taken to
infinity,  we see an effect from the tail of the distribution of the quenched mean
$\mu_0$.   We illustrate this with the case of exponential $\{F_j\}$.
Now the number $n\alpha$ of   distributions $F_j$ is small compared to the
number $n$ of weights $X(i,j)$ in each row, hence   the fluctuations among the $F_j$'s
become prominent.  The effect comes in two forms:  first, the leading term is no longer the
averaged mean $\mu$ but the maximal mean.  Second,
  if large values among the row means $\mu_j$ are rare,
  the order of the $\alpha$-dependent correction is smaller than the $\sqrt\alpha$
  seen above and this order of magnitude  depends on the
tail of the distribution of $\mu_0$.  As an exponent characterizing this tail changes,
we can see a phase transition of sorts in the power of $\alpha$, with a logarithmic
correction at the transition point.

For  general distributions  we derive bounds on $\Psi(1, \alpha)$ that indicate that
in the case of finitely many distributions the correction is of order  $\sqrt\alpha$.

As auxiliary results we  need bounds on the limits for last-passage models
with Bernoulli weights under a random environment.      However, with Bernoulli weights
the standard  corner growth  model    is {\sl not} one of the explicitly
solvable cases.
 The model with  Bernoulli weights does become
 solvable   when the path geometry is altered suitably.  The model  we
take up is the one where the paths are weakly increasing in one
coordinate but strictly in the other.  There are two cases,
depending on which coordinate is required to increase strictly.   If
we require the  $x$-coordinate to increase strictly then an
admissible path $\{(x_0,y_0), (x_1,y_1),\dotsc, (x_m,y_m)\}$
satisfies
\be \text{$x_{i+1}-x_i =1$ 
and $y_0\le y_1\le  \dotsm\le
y_m$. } \label{path-2}\ee
The other  case   interchanges $x$ and $y$. These
cases have to be addressed separately because the random environment
attached to rows makes the model asymmetric. The sum of these two
last-passage values gives   a bound for the case where neither
coordinate is required to increase strictly in each step.

We derive   the exact limit constants for   Bernoulli models with both types of
 strict/weak paths.   For one of them this has been done
  before by Gravner, Tracy and Widom \cite{GTW2002}.  Their proof utilizes the
  fact that the distribution of $T(k,\ell)$ is a symmetric function of the
  environment (at least for the particular Bernoulli case they study).
Our proof is completely different. It is based on the idea   in \cite{Timo1998}  where
the limit for
   the homogeneous case was derived: the last-passage model is coupled
with a particle system whose invariant distributions can be written down explicitly,
and then through some convex analysis the speed of a tagged particle yields
the explicit limit of the last-passage model.

\medskip

{\sl Further remarks on the literature.}  The present paper does not address questions of
fluctuations, but let us mention some highlights from the literature.
For the last-passage model with i.i.d.\ exponential or geometric weights,  the
distributional limit with fluctuations of order $n^{1/3}$ and limit
given by the  Tracy-Widom GUE distribution was proved by Johansson \cite{joha}.
As for the shape,  universality has been achieved only close to the boundary,
by  Baik-Suidan \cite{baik-suid}  and  Bodineau-Martin \cite{bodi-mart}.

Fluctuations of the Bernoulli model with strict/weak paths and
homogeneous weights were derived first in
\cite{joha01} and later also in \cite{grav-trac-wido-01}.  For the model in a
random environment fluctuation limits appear  in \cite{GTW2002, grav-trac-wido-02b}.

On the lattice $\bZ_+^2$ we can imagine three types of nondecreasing
paths: (i) weak-weak: both coordinates required to increase weakly, the
type used in \eqref{last-pass-1};  (ii) strict-weak:  one coordinate
increases strictly, as above in \eqref{path-2};  and (iii) strict-strict:
 both coordinates increase
strictly so an admissible path $\{(x_0,y_0), (x_1,y_1),\dotsc, (x_m,y_m)\}$
satisfies
$x_0<\dotsm<x_m$ and $y_0< \dotsm<
y_m$.   As mentioned, with Bernoulli weights the strict-weak case
is solvable but the weak-weak case appears harder. The third case,
strict-strict, is also solvable
 with Bernoulli weights.   The shape was derived in
\cite{sepp97incr}  and recent work on this model appears in  \cite{georgiou-10}.

\medskip

{\sl Organization of the paper.}   The main results on the shape
close to the boundary are in Section \ref{sec-main} and the results
for Bernoulli models in Section \ref{sec:bernoulli}.  Section
\ref{sec:limit} sketches the proof of the existence of the  limiting
shape, a result we basically take for granted.   The main proofs
follow:  in Section \ref{sec:thm:main} for Theorem \ref{thm:main} on
$\Psi(\alpha, 1)$, in Section \ref{sec:1,a-thm} for Theorem
\ref{1,a-thm} on $\Psi(1, \alpha)$,  and in Section
\ref{sec:exp-thm} for Theorem \ref{exp-thm} for the exponential
model.

\medskip

{\sl Some frequently used notation.}  We write \[
\underset{\bP}{\esssup}\,f=\inf\{ s\in\bR:  \bP(f>s)=0\}\] for the
essential supremum of a function $f$ under a measure $\bP$.
 $\bZ_+=\{0,1,2,\dotsc\}$,  $\bN=\{1,2,3, \dotsc\}$,  and $\bR_+=[0,\infty)$.
  $I(A)$ is the indicator function of event $A$.

\section{Main results}\label{sec-main}
First a precise definition of the  last-passage model in a random environment.
Let $\bP$ be a stationary,
ergodic  probability measure on the space
$\mathcal{M}_1(\bR)^{\bZ_+}$ of sequences of Borel probability
distributions on $\bR$. $\bE$ denotes expectation under $\bP$. For
some of the main results $\P$ will   be assumed to be an i.i.d.\ product
measure.   A realization of the distribution-valued process under
$\P$ is denoted by $\{F_j\}_{j\in\bZ_+}$. This is the environment.   Given $\{F_j\}$, the
weights  $\{X(z): z \in \bZ_+^2\}$ are independent real-valued random variables
with marginal
distributions $X(i,j)\sim F_j$ for $(i,j)\in\bZ_+^2$. Let $(\Omega,
\mathcal{F}, \mathbf{P})$ be the probability space on which all
variables $\{ F_j, X(i,j)\}$ are defined, and denote expectation under
$\mP$ by $\mE$.

A (weakly)  nondecreasing path is a sequence of
points  $z_0=(x_0,y_0),
z_1=(x_1,y_1),\dotsc, z_m=(x_m, y_m)$ in $\bZ_+^2$
 that satisfy $x_0\leq  x_1 \leq
\dotsm \leq x_m$, $y_0 \leq y_1 \leq \dotsm \leq y_m$, and
$|x_{i+1} - x_i| + |y_{i+1} -y_i| =1$.
 For $z_1, z_2 \in \bZ^2_+$ with $z_1 \leq
z_2$ (coordinatewise ordering), let $\Pi(z_1, z_2)$ be the set of nondecreasing paths from
$z_1$ to $z_2$.   Whether   the endpoints $z_1$ and  $z_2$ are
  included in the path  makes no difference to the limit results below.
 The last-passage time $T(z_1, z_2)$
 from $z_1$ to $z_2$ is defined  by
$$T(z_1, z_2) = \max_{\pi \in \Pi(z_1,z_2)}  \sum_{ z\in \pi} X(z).$$
When $z_1 = 0$ abbreviate
 $ \Pi(z) = \Pi(0, z)$ and $ T(z) = T(0, z)$.

Put these three assumptions on the model:
\be \mE |X(z)| < \infty,  \label{momass}\ee
\be
\int_0^{\infty} \Bigl\{1-\bE( F_0(x))\Bigr\}^{1/2} dx < \infty,
\label{tailass1}\ee
and
\be
\int_0^{\infty} \underset\bP{\esssup}(1 -  F_0(x))\,dx<\infty.
\label{tailass2}\ee
We begin with this by now standard result that defines our object of study, namely
the function $\Psi$.  The proof is briefly commented on  in Section \ref{sec:limit}.

\begin{prop}  Assume $\P$ is ergodic
and satisfies  \eqref{momass}, \eqref{tailass1} and \eqref{tailass2}.
Then for all $(x,y)\in(0,\infty)^2$ the  last passage time constant
\begin{equation}
\label{def1}
 \Psi(x,y) = \lim_{n\rightarrow \infty}\frac{1}{n} T(\lfloor
nx \rfloor, \lfloor ny \rfloor)
\end{equation}
exists as a limit both $\mP$-almost surely and in $L^1(\mP)$.
Furthermore, $\Psi(x,y)$ is a homogeneous, concave and continuous
function on $(0,\infty)^2$.  \label{pr:limit}
\end{prop}

Assumption \eqref{tailass1}  is also used for the
constant distribution case, see (2.5) in \cite{Martin2004}.
Some further control along the lines of assumption
 \eqref{tailass2} is required for our case.  For example,
suppose $1-F_j(x)=e^{-\xi_j x}$ for random $\xi_j\in(0,\infty)$.
Then  \eqref{tailass2} holds iff ${\essinf}_\bP (\xi_0)>0$.
If the distribution of  $\xi_0$ is not bounded away from zero,
  $n^{-1}T(n,n)\to\infty$ because
we can  simply collect all the weights from the row with minimal
$\xi_j$ among $\{\xi_0,\dots,\xi_n\}$. However, assumption
\eqref{tailass1} can be satisfied without bounding $\xi_0$ away from
zero.

Now we turn to the main results of the paper on   the   form of  the limit shape at
the boundary.   As explained in the introduction, for
 $\Psi(\alpha,1)$  we find a universal form as  $\alpha\searrow 0$.
In addition to the earlier assumptions, we need similar control of the left tail of the
distributions: 
\be \int_{-\infty}^0 \bigl(\bE [F_0(x)]\bigr)^{{1}/{2}} dx< \infty
\label{tailass3}\ee and \be \int_{-\infty}^0 \underset\bP{\esssup}\,
F_0(x)\,dx<\infty.
 \label{tailass4}\ee
Let us point out
  that \eqref{tailass1} and \eqref{tailass3} together
guarantee   $\mE |X(z)|^2 < \infty$.   Let  $\mu_j = \mu(F_j)$ and $\sigma_j^2 =
\sigma^2(F_j)$
 denote the mean and variance of
distribution $F_j$. These are random variables under $\P$ with
expectations
  $\mu = \E( \mu_0)$ and   $\sigma^2 =\E(\sigma^2_0)$.

\begin{thm}  Assume the process $\{F_j\}$ is i.i.d.\ under $\P$, and satisfies 
tail assumptions \eqref{tailass1},   \eqref{tailass2},
\eqref{tailass3} and    \eqref{tailass4}. Then,  as
$\alpha\downarrow 0$, $\Psi(\alpha,1) = \mu + 2\sigma \sqrt{\alpha}+
o(\sqrt{\alpha\,}).$ \label{thm:main}\end{thm}

Assumptions \eqref{tailass1} and    \eqref{tailass3}
 are direct counterparts of what was used for  Theorem 2.4
in \cite{Martin2004}.   Assumptions \eqref{tailass2} and
\eqref{tailass4}  are additional assumptions needed for handling the
random environment.  These assumptions are used to control estimates that come
from bounding limits of Bernoulli models.

\bigskip

We turn to
the case  $\Psi(1,\alpha)$.    The results will be
  qualitatively  different from  Theorem \ref{thm:main}.
The leading term will be  the essential supremum of the mean   instead of the averaged mean and
   we will see different orders for the first $\alpha$-dependent correction term.

  First a general result for which we restrict ourselves
to the case of finitely many distributions, but we can relax the i.i.d.\ assumption
of the random distributions.

\begin{thm}
Assume  the process $\{F_j\}$ of probability distributions is
stationary, ergodic, and  has a state space of finitely many
distributions $H_1,\dotsc, H_L$  each of which satisfies Martin's
\cite{Martin2004} hypothesis \be  \int_0^\infty
(1-H_\ell(x))^{1/2}\,dx  +   \int_{-\infty}^0
H_\ell(x)^{1/2}\,dx<\infty. \label{martin8}\ee Let $\mu^*=\max_\ell
\mu(H_\ell)$ be the maximal mean of the $H_\ell$'s.
 Then there exist
  constants $0<c_1<c_2<\infty $   such that,  as $\alpha\downarrow 0$,
\be   \mu^* + c_1  \sqrt{\alpha}+ o(\sqrt{\alpha}\,)\; \le \; \Psi(1, \alpha) \; \le \; \mu^* + c_2  \sqrt{\alpha}+ o(\sqrt{\alpha}\,).  \label{1,a-bd}\ee
\label{1,a-thm}\end{thm}
We would expect  $c_1=c_2$  but our proof does not give it.

\medskip

Finally, we consider the case $\Psi(1,\alpha)$ for   the exponential model
where  some (partially) explicit calculation is possible.
 Here we see
how the tail of the random mean $\mu_0$ creates different orders
of magnitude for the   $\alpha$-dependent correction term.
   Let
  $\{\xi_j\}_{j\in\bZ_+}$ be an i.i.d.\ sequence of
random variables that satisfy  $0<c\leq \xi_j$ with common distribution $m$.
To distinguish the exponential model from the general one we write
 $G_j(x)=1-e^{-\xi_jx}$ for   the distribution function of the exponential
 distribution with parameter $\xi_j$,  and $\Psi_G$ for the limiting time constant.
 We
assume $c$ is the exact lower bound: $m[c,c+\e)>0$ for each $\e>0$.
Then the essential supremum of the random  mean is $\mu^*=c^{-1}$.

An implicit description of the limit shape was derived in
\cite{TimoKrug} by way of studying an exclusion process with random
jump rates attached to particles. We recall the result here. One
explicit shape is needed for the proof of Theorem \ref{thm:main}
also, so this result will serve there too.

 Define first a critical value
  $u^* = \int_{[c,\infty)}
\frac{c}{ \xi -c }\, m(d\xi)\in(0,\infty]$. For $0 \leq u < u^*$
define $a = a(u)$ implicitly by
      $$u =  \int_{[c,\infty)} \frac{a}{ \xi-a}\,m(d\xi). $$
$a(u)$ is strictly increasing, strictly concave,
continuously differentiable and one-to-one from
$0< u < u^*$ onto $0< a <c$. We let $a(u) = c$ for $u\geq u^*$
if $u^* < \infty$.
Then define  $g:\bR_+\to\bR_+$ by
\be g(y) = \sup_{u\geq 0}\{ -yu + a(u)\}, \quad y\ge0. \label{ga-dual}\ee
The function $g$  is monotone decreasing,  continuous,
and $g(y)=0$ for $y\ge a'(0+)=1/\mu_G$.  It is the level curve of the time constant.
The equations connecting the two are $g(y)=\inf\{ x>0: \Psi_G(x,y)\ge 1\}$
and
  \be \Psi_G(x,y) = \inf\{t \geq 0: tg({y}/{t}) \geq x\}. \label{PsiGg}\ee

Qualitative properties of the limit shape depend on the tail of the
distribution $m$ at $c+$, and transitions occur where the integrals
$\int_{[c,\infty)}  (\xi - c)^{-2}\, m(d\xi)$ and $\int_{[c,\infty)}
(\xi- c)^{-1}\, m(d\xi)$ blow up.  (For details see
\cite{TimoKrug}.) These same regimes appear in our results below.
For the case $\int_{[c,\infty)}(\xi - c)^{-2}\, m(d\xi) = \infty$ we
make a precise assumption about the tail of the distribution of the
random rate: \be \exists \; \nu\in[-1,1], \; \kappa>0 \ \ \text{such
that} \ \ \lim_{\xi \searrow c}
\frac{m[c,\xi)}{(\xi-c)^{\nu+1}}=\kappa. \label{exp-ass1}\ee The
value $\nu=-1$ means that the bottom rate $c$ has probability
$m\{c\}=\kappa>0$. Values $\nu<-1$ are of course not possible.

\begin{thm}  For the model with exponential distributions with i.i.d.\ random
rates the limit $\Psi_G$ has these asymptotics close to the $x$-axis.

{\sl Case 1:}   $\int_{[c,\infty)}  (\xi - c)^{-2}\, m(d\xi) <
\infty$. Then there exists $\alpha_0>0$ such that
  \be \Psi_G(1,\alpha) = {c}^{-1} + \alpha \int_{[c,\infty)}
\frac{1}{\xi-c} \,m(d\xi) \quad \text{for $\alpha\in[0,\alpha_0]$.}
\label{expcase1} \ee

{\sl Case 2:} \eqref{exp-ass1}  holds so that, in particular
$\int_{[c,\infty)}  (\xi - c)^{-2}\, m(d\xi) = \infty$. Then as
$\alpha \searrow 0$,
\begin{align}
&\text{if $\nu \in (0,1]$ then} \  \ \Psi_G(1,\alpha) = {c}^{-1}+
\alpha \int_{[c,\infty)} \frac{1}{\xi - c} \,m(d\xi) +o(\alpha)\,;
\label{casenu01} \\
  &\text{if $\nu =0$ then} \qquad   \Psi_G(1,\alpha) =  c^{-1} -\kappa\alpha\log\alpha
  +o(\alpha\log\alpha) \,;  \label{casenu0} \\
 &\text{if $\nu \in [-1,0)$ then} \   \
 \Psi_G(1,\alpha) = c^{-1} + B\alpha^{\frac1{1-\nu}} +o(\alpha^{\frac1{1-\nu}}).   \label{casenu-10}
 \end{align}
\label{exp-thm}\end{thm}
In statement \eqref{casenu-10} above $B=B(c,\kappa,\nu)$ is a constant
whose explicit definition is in equation \eqref{exp18} in the proof section below.
The extreme case $\nu=-1$ is the one that matches up with Theorem \ref{1,a-thm}.

For some heuristic understanding of Theorem \ref{exp-thm} we turn to the queueing interpretation
discussed in the Introduction.
 Quantity  $n\Psi(1,\alpha)$ represents the time when customer $n$ departs from
server $\fl{n\alpha}$ (rigorously speaking in the $n\to\infty$ limit),  when initially all customers are queued up at server $0$.
When  $u^*<\infty$ (Case 1 and subcase $\nu>0$ from Case 2) the results of
\cite{andj-etal} suggest that, at some distance but not too far from the first queue,
 the   system should converge to an equilibrium
where the queue length at server $j$  (whose service rate is
$\xi_j$) is geometric with mean $c/(\xi_j-c)$  and the departure
process from each queue has rate $c$. A customer arriving at a queue
with this geometric number of customers present  spends on average
time $1/(\xi_j-c)$  at that queue.  In Case 1 this picture is
precise enough so that equation \eqref{expcase1} can be   naively
understood in these terms:     a single customer travels through
$\fl{n\alpha}$ servers   in  time $  n\alpha \int_{[c,\infty)}
\frac{1}{\xi-c} \,m(d\xi)$. After this it takes another $n/c$ time
to see $n$ customers go through server $\fl{n\alpha}$.  Together
these terms make up the right-hand side of \eqref{expcase1}.   This
argument requires  a high density of customers  because once the
customer density drops below $u^*$ the system chooses an equilibrium
with flow rate below $c$.  (Again, for details we refer to
\cite{andj-etal}.)

This point can be detected in the proofs for Case 2  where we find a parameter $a_0$ that
in some sense represents a  flow rate and replaces $c$ in \eqref{expcase1}
(see eqn.~\eqref{case2}).
The formulas in Case 2 are then obtained by   estimating   $c-a_0$.

\section{Bernoulli   models with strict-weak paths  in a random environment} \label{sec:bernoulli}

This section  looks at   last-passage models with Bernoulli-distributed
weights.  The environment is now an i.i.d.\ sequence $\{p_j\}_{j\in\bZ_+}$ of
numbers $p_j\in[0,1]$, with distribution $\bP$.  Given $\{p_j\}$,  the weights $\{X(i,j)\}$ are independent
with marginal distributions
$P(X(i,j)=1)=p_j=1-P(X(i,j)=0)$.  We consider two last-passage times
that differ by the type of admissible path:  for $z_1, z_2\in\bZ_+^2$
\be
T_{\rightarrow}(z_1, z_2) = \max_{\pi \in \Pi_{\rightarrow}(z_1,z_2)}  \sum_{ z\in \pi} X(z)
\quad\text{and}\quad
T_{\uparrow}(z_1, z_2) = \max_{\pi \in \Pi_{\uparrow}(z_1,z_2)}  \sum_{ z\in \pi} X(z).
\label{bernT}\ee
 In terms of coordinates denote the endpoints by  $z_k=(a_k,b_k)$, $k=1,2$.  Then admissible paths
$\pi\in\Pi_{\rightarrow}(z_1,z_2)$ are  of the form
$\pi=\{(a_1,y_0), (a_1+1,y_1), (a_1+2,y_2),\dotsc,
(a_2,y_{a_2-a_1})\}$ with $b_1\le y_0\le y_1\le \dotsm\le
y_{a_2-a_1}\le b_2$,  while paths $\pi\in \Pi_{\uparrow}(z_1,z_2)$
are of the form $\pi=\{(x_0, b_1), (x_1, b_1+1),\dotsc,
(x_{b_2-b_1}, b_2)\}$ with $a_1\le x_0\le x_1\le \dotsm\le
x_{b_2-b_1}\le a_2$. Thus paths in $\Pi_{\rightarrow}(z_1,z_2)$
increase strictly in the $x$-direction while those in
$\Pi_{\uparrow}(z_1,z_2)$    increase strictly in the $y$-direction.
The last-passage times   $T_{\rightarrow}(z_1, z_2)$  and
$T_{\uparrow}(z_1, z_2)$ record the maximal weights of such paths in
the lattice rectangle $([a_1,a_2]\times[b_1,b_2])\cap\bZ_+^2$. The
following figures illustrate the two types of admissible paths when
we take $z_1 = (0,0)$ and $z_2 = (5,5)$:

\begin{figure}[h!]
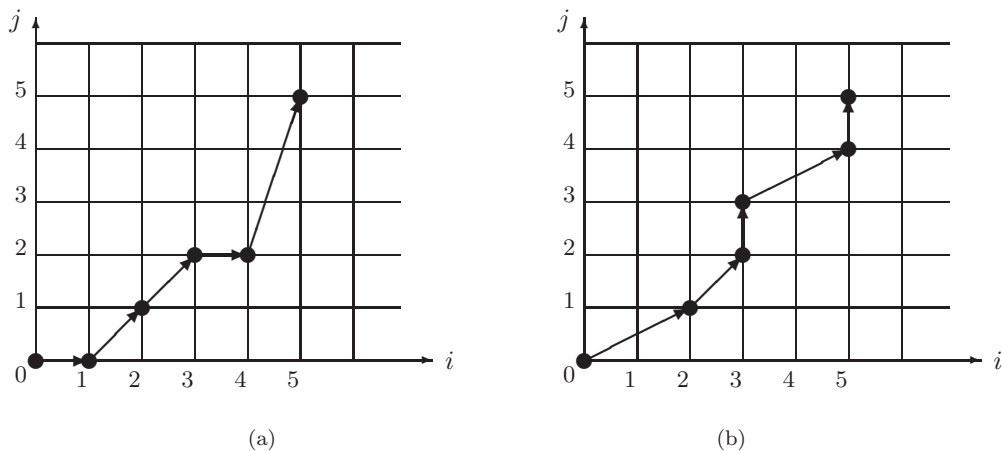

 \subfigure[]{\put(-80,20){\vector(1,0){150}}
\put(-80,20){\vector(0,1){130}}
\multiput(-80,40)(0,20){6}{\line(1,0){138}}
\multiput(-60,20)(20,0){6}{\line(0,1){120}}
\multiput(-80,20)(20,0){2}{\circle*{6}}
\multiput(-40,40)(0,20){1}{\circle*{6}}
\multiput(-20,60)(20,0){2}{\circle*{6}} \put(20,120){\circle*{6}}
\thicklines \multiput(-80,20)(20,0){1}{\vector(1,0){20}}
\put(-60,20){\vector(1,1){20}} \put(-40,40){\vector(1,1){20}}
\put(-20,60){\vector(1,0){20}} \put(0,60){\vector(1,3){20}}
\put(74.7,17){$i$} \put(-90,146.5){$j$} \put(-88,12){\small 0}
\put(-65,10){\small 1} \put(-45,10){\small 2} \put(-25,10){\small 3}
\put(-5,10){\small 4}\put(15,10){\small 5} \put(-88,40){\small 1}
\put(-88,60){\small 2} \put(-88,80){\small 3} \put(-88,100){\small
4}\put(-88,120){\small 5}} \hspace{60mm} \subfigure[]{
\put(-50,20){\vector(1,0){150}} \put(-50,20){\vector(0,1){130}}
\multiput(-50,40)(0,20){6}{\line(1,0){138}}
\multiput(-30,20)(20,0){6}{\line(0,1){120}}
\multiput(-50,20)(20,0){1}{\circle*{6}}
\multiput(-10,40)(0,20){1}{\circle*{6}}
\multiput(10,60)(20,0){1}{\circle*{6}} \put(10,80){\circle*{6}}
\put(50,100){\circle*{6}} \put(50,120){\circle*{6}} \thicklines
\put(-50,20){\vector(2,1){40}}
 \put(-10,40){\vector(1,1){20}}
\put(10,60){\vector(0,1){20}} \put(10,80){\vector(2,1){40}}
\put(50,100){\vector(0,1){20}} \put(104.7,17){$i$}
\put(-60,146.5){$j$} \put(-58,12){\small 0} \put(-35,10){\small 1}
\put(-15,10){\small 2} \put(5,10){\small 3} \put(25,10){\small
4}\put(45,10){\small 5} \put(-58,40){\small 1} \put(-58,60){\small
2} \put(-58,80){\small 3} \put(-58,100){\small
4}\put(-58,120){\small 5}} \caption{Admissible paths in
$\Pi_\rightarrow(z_1,z_2)$ and $\Pi_\uparrow(z_1,z_2)$}
\label{bernpath1}
\end{figure}%

 As before we simplify   notation with  $T_{\rightarrow}(0, z)= T_{\rightarrow}(z)$.
  The almost sure limits are denoted by
\be
 \Psi_{\rightarrow}(x,y) = \lim_{n\to \infty}\frac{1}{n} T_{\rightarrow}(\lfloor
nx \rfloor, \lfloor ny \rfloor)
\quad\text{and}\quad
 \Psi_{\uparrow}(x,y) = \lim_{n\to \infty}\frac{1}{n} T_{\uparrow}(\lfloor
nx \rfloor, \lfloor ny \rfloor)
\label{bernPsi}\ee
 for $(x,y)\in(0,\infty)^2$.
The existence of the limits needs no further comment.

The next theorem gives the
explicit  limits.  \eqref{bernoulli5} is the same as   in \cite[Thm.~1]{GTW2002}.
Inside the $\bE(\,\dotsm)$ expectations below $p$ is the random
Bernoulli probability.  Let $b=\underset{\bP}{\esssup}\,p$  denote the maximal probability.

\begin{thm}
The limits in \eqref{bernPsi} are as follows for $x,y\in(0,\infty)$.
  \be \label{bernoulli5} \Psi_{\rightarrow} (x,y) =
\begin{cases}
 bx + y(1-b) \bE \bigl(  \frac{p }{b-p}\bigr), & {x}/{y} \ge
\bE\bigl( \frac{p(1-p) }{(b - p)^2}\bigr)\\[6pt]
 yz_0^2\bE\bigl(  \frac{1-p }{(z_0 -  p)^2} \bigr) -y,
 &\bE\bigl( \frac{p }{1-p}\bigr) < {x}/{y} < \bE\bigl( \frac{p(1-p) }{(b-p)^2}\bigr)\\[6pt]
x, &  0< {x}/{y} \le \bE\bigl( \frac{p }{1-p}\bigr)
\end{cases}
\ee with $z_0\in(b,1)$ uniquely defined by the equation
$${x}/{y}=   \,\bE \Bigl[ \frac{  p(1-p)}{(z_0-  p)^2}\Bigr].$$

\be \label{bernoulli6} \Psi_{\uparrow} (x,y) = \begin{cases}
y-y z_0^2
\bE\bigl( \frac{1-p}{(z_0 +p)^2} \bigr), & 0 < {x}/{y} <
\bE\bigl(\frac{1-p}{p} \bigr)\\[6pt]
y, & {x}/{y} \ge \bE\bigl( \frac{1-p}{p} \bigr)
\end{cases}\ee
with $z_0\in(0,\infty)$ uniquely defined by the equation
$${x}/{y} = \bE \Bigl[   \frac{p(1-p)}{(z_0 +p)^2} \Bigr]. $$
\label{thm:bernoulli}\end{thm}

Our second result gives  simplified  bounds   that are    useful for the proof
of the main result Theorem \ref{thm:main}.    Let $\bar p=\bE(p)$ be the mean of the environment.
  $\Psi(x,y)$ is the limiting time constant with weakly increasing paths defined
in Proposition \ref{pr:limit}.

\begin{thm}
The following three inequalities hold for the Bernoulli model:
\begin{equation}
\label{bernoulli1} \Psi_{\rightarrow} (x,y)  \leq bx + 2
\sqrt{\bar{p}(1-b) xy },
\end{equation}
\begin{equation}
\label{bernoulli2}
 \Psi_{\uparrow} (x,y) \leq \bar{p}y + 2 \sqrt{\bar{p}(1-\bar{p})xy}
\end{equation}
and
\begin{equation}
\label{bernoulli3}
 \Psi (x,y)\leq \bar{p}y + 4 \sqrt{\bar{p}(1-\bar{p})xy} + bx.
\end{equation}
\label{thm:bernoulli2}\end{thm}

\eqref{bernoulli3}  follows from   \eqref{bernoulli1} and \eqref{bernoulli2}   because
$ \Psi (x,y) \leq \Psi_{\rightarrow} (x,y)+\Psi_{\uparrow} (x,y)$.
 Another loose estimate we will use later following
\eqref{bernoulli3} is
\begin{equation}
\label{bernoulli4} \Psi (x,y) \leq \bar{p}y + 4
\sqrt{\bar{p}(1-\bar{p})xy} + bx \leq (y+4\sqrt{xy})\sqrt{\bar{p}} +
bx.
\end{equation}

We prove the formulas and inequalities first for  $\Psi_{\rightarrow}$ and then
for $\Psi_{\uparrow}$.  For some parts of  the proofs it is convenient to  assume   $b< 1$.
 Results for the case  $b=1$ follow by taking a limit.

 \begin{proof}[Proof of \eqref{bernoulli5} and \eqref{bernoulli1}]
 We   adapt    the proof  from  \cite{Timo1998} to the random environment situation
and sketch the main points.

Consider now the environment $\{p_j\}$ fixed, but the weights $X(i,j)$ random.
For integers $0\le s<t$ and $a, k$ define  an inverse to the last passage time as
 $$\Gamma ((a,s), k, t) = \min\{l \in \bZ_+: T_{\rightarrow}((a+1,s+1),(a+l, t)) \geq k \}.$$
 Note that $\Gamma ((a,s), 0, t) =0$ but $\Gamma ((a,s), k, t) >0$ for $k>0$.
 Knowing the limits of the variables $\Gamma$ is the same as knowing
 $\Psi_\rightarrow$.  By the homogeneity of $\Psi_\rightarrow$  it is enough to find $h(x)=\Psi_\rightarrow(x,1)$.  By the homogeneity and superadditivity of  $\Psi_\rightarrow$,
 $h$ is concave and nondecreasing.
Let $g$ be the inverse function of $h$ on $\bR_+$.   Then $g$ is convex
and nondecreasing, and
\[  tg(x/t)= \lim_{n\to\infty} \frac1n \Gamma((0,0), \fl{nx}, \fl{nt} ).  \]

To find these functions
we construct an exclusion-type  process $z(t)=\{z_k(t): k\in\bZ\}$ of labeled, ordered  particles
$z_k(t)<z_{k+1}(t)$ that jump leftward on the lattice $\bZ$, in discrete time $t\in\bZ_+$.
Given an initial configuration  $\{z_i(0)\}$ that satisfies
$z_{i-1}(0) \leq z_i(0)-1 $  and
 $\liminf_{i \rightarrow -\infty} |i|^{-1} z_i(0) > -1/b$,  the evolution is defined by
 \be z_k(t) = \inf_{i: i
\leq k} \{ z_i(0)  + \Gamma((z_i(0),0),k-i, t )\}, \quad k\in\bZ, \, t\in\bN. \label{bernz}\ee
It can be checked that  $z(t)$ is a  well-defined Markov process, in particular that
  $z_k(t) > -\infty$ almost surely.

Define the process $\{\eta_i(t)\}$  of interparticle distances by
  $\eta_i (t)= z_{i+1}(t) -z_{i}(t)$ for   $i\in \bZ$ and $t\in
\bZ_+$. By Prop.~1 in \cite{Timo1998} process $\{\eta_i(t)\}$  has a family of
i.i.d.\ geometric invariant distributions indexed by the mean  $u\in [1, {b}^{-1})$
and defined by
\be P(\eta_i = n)  = u^{-1} (1-u^{-1})^{n-1}, \quad n\in\bN. \label{berninv}\ee
 Let  $x_k(t) = z_k(t-1) - z_{k}(t)\ge 0$  be the absolute size of the jump of
the $k$th particle from time $t-1$ to $t$, and let $q_t=1-p_t$.   From
(6.5) in \cite{Timo1998}, in the stationary process \be P(x_k(t) = x) =
\begin{cases}
        (1-up_t) q_t^{-1} &
        x=0\\
         p_t(1-up_t) q_t^{-1}(u-1)^x(uq_t)^{-x} & x=1,2,3,\dotsc\\
\end{cases}\label{bern-step-distr}\ee

We track the motion of particle $z_0(t)$ in a stationary situation.  The initial state
is defined by setting
$z_0(0)=0$ and by letting $\{\eta_i(0)\}$ be i.i.d.\ with common distribution \eqref{berninv}.
With $k=0$, divide by $t$ in \eqref{bernz}  and take   $t\to\infty$.  Apply laws
of large numbers inside the braces in \eqref{bernz}, with some simple estimation to pass the limit
through the infimum, to find  the average speed of the tagged particle:
\be -\, \lim_{t \rightarrow \infty} \frac{1}{t}  z_0(t)
 =\sup_{x\ge 0} \{ ux-g(x)\} \equiv f(u). \label{bern-lim-8}\ee
The last equality defines the speed $f$ as
   $f=g^+$, the {\sl monotone conjugate} of $g$.
   It is natural to set $f(u)=0$ for $u\in[0,1)$,  $f(b^{-1})=f((b^{-1})-)$,  and $f(u)=\infty$ for $u>b^{-1}$.
  By \cite[Thm.~12.4]{rock-ca}
\be
 g(x) = \sup_{u \geq 0} \{xu - g^+(u)\} = \sup_{1\leq u \leq {1}/{b}} \{xu - f(u)\}.
 \label{berndual} \ee

Since $z_0(t)$ is a sum of  jumps $x_0(k)$ with distribution \eqref{bern-step-distr} we have
the second moment bound $\sup_{t\in\bN}  \mE[(t^{-1} z_0(t))^2]<\infty$, and
consequently the limit  in \eqref{bern-lim-8} holds also in expectation.
From this
\be\begin{aligned}
f(u)&=  -\,\lim_{t \rightarrow \infty}\mE[  {t}^{-1}  z_0(t)]
= \lim_{t \rightarrow \infty} \mE\Bigl[\,  {t}^{-1}\sum_{k=1}^{t} x_0(k)\,\Bigr] = \mE[  x_0(0)] \\
&=\bE  \sum_{x=1}^\infty x(u-1)^x (uq)^{-x} p (1-up) (1-p)^{-1}
= \bE \Bigl[ \frac{pu(u-1)}{1-up} \Bigr] .
\end{aligned}\label{bernf}\ee

We find $g(x)$ from \eqref{bernf} and \eqref{berndual}:
 \be\label{gx2} g(x) = \begin{cases}
 {x}/{b} - b^{-1}(1-b) \bE\frac{ p}{(b-p)}  & x \ge b^2 \bE
\frac{ (1-p)}{(b - p)^2}-1\\[5pt]
  u_0^2 \bE\frac{
p(1-p)}{(1- u_0 p)^2}  &  \bE  \frac{ p}{1-p} < x <
b^2\bE \frac{ (1-p)}{(b - p)^2} -1\\[5pt]
x &  0<  x <\bE \frac{ p}{1-p}
\end{cases} \ee
where $u_0 \in (1,{b}^{-1})$ is uniquely defined by the equation
$ x+1 =\bE{(1-p)}{(1 - u_0 p)^{-2}}$.
From this we find  the inverse function $h(x)= g^{-1}(x)$  and then
$\Psi_{\rightarrow} (x,y) = y h( {x}/{y})$.  We omit these details and consider  \eqref{bernoulli5}
proved.

To  prove \eqref{bernoulli1} we return to the duality \eqref{berndual} and write
\be  g(x)
  \geq \sup_{1 \leq u < 1/b}\{xu -
\tilde f(u)\}\quad \text{for}\quad \tilde f(u) = \frac{u(u-1)}{1-ub} \bar{p}. \label{bern11}\ee
 $\tilde f'(u) = x$ is solved by  $u^* = b^{-1}\Bigl({1 - \sqrt{\frac{(1-
b)\bar{p}}{ bx+\bar{p}}}}\Bigr). $

 When $x \geq \frac{\bar{p}}{1-b}$,  we have $u^* \in [1,\frac{1}{b})$, and then
 \[\begin{split}
g(x)  \geq  xu^* - \tilde f(u^*)
       =  \frac{1}{b^2} \bigl(\sqrt{(1-b)\bar{p}} - \sqrt{bx+\bar{p}}\,\bigr)^2.
      \end{split}\]
Consequently
\[\begin{split} g^{-1}(x) &\leq \frac{1}{b} \bigl(\sqrt{b^2 x}
+\sqrt{(1-b)\bar{p}}\,\bigr)^2-
\frac{\bar{p}}b  = bx -\bar{p} + 2 \sqrt{(1-b)\bar{p} x}.
 \end{split}\]
When $x < \frac{\bar{p}}{1-b}$,  the supremum in \eqref{bern11} is
attained at $u = 1$, and in this case  $$g^{-1}(x)\le x  \leq bx + 2
\sqrt{(1-b) \bar{p} x}.$$ The bound \eqref{bernoulli1} now follows
from $\Psi_{\rightarrow} (x,y) = y g^{-1}({x}/{y}).$
\end{proof}

\begin{proof}[Proof of \eqref{bernoulli6} and \eqref{bernoulli2}]
The scheme is the same, so we omit some more details.
The inverse of the last-passage time is now defined
\[ \Gamma ((a,s), k, t) = \min\{l \in \bZ_+: T_{\uparrow}((a,s+1),(a+l, t)) \geq k \}. \]
Vertical distance $t-s$ allows for at most $t-s$ marked points, so the above
quantity must be set equal to $\infty$ for $k>t-s$.
The particle process $\{z(t): t\in\bZ_+\}$ is  defined by the same formula \eqref{bernz} as before but it
is qualitatively different.  The particles still jump to the left,
but  the ordering rule is now $z_k(t)\le z_{k+1}(t)$
so particles are allowed to sit on top of each other.  Well-definedness of the dynamics
needs no further restrictions
on admissible particle configurations   because the minimum in
\eqref{bernz} only considers $i\in\{k-t,\dotsc, k\}$ so it is well-defined for
all initial configurations $\{z_i(0):i\in\bZ\}$ such that $z_i(0)\le z_{i+1}(0)$.

The following can be checked.  Under a fixed environment $\{p_j\}$,
the gap process $\{\eta_i(t)=z_{i+1}(t)-z_i(t): i\in\bZ\}$ has i.i.d.\ geometric
invariant distributions
$ P(\eta_k =n) =( \frac{1}{1+u}) (\frac{u}{1+u})^n$, $n\in\bZ_+$,
indexed by the mean $u\in\bR_+$.   In this stationary situation the
successive jumps   $x_k(t) = z_k(t-1) - z_{k}(t)$ of a tagged particle  have
distribution
\[ P(x_{k}(t) =y )=
\begin{cases}
\frac{1}{1+up_t} &y=0 \\[4pt]
(\frac{u}{u+1})^y \frac{p_t}{1+up_t} & y \geq 1.
\end{cases}
\]
From here the analysis proceeds the same way as for the other model.  The speed
function is defined by
 \begin{align*}
f(u)&= -\lim_{n \rightarrow \infty} \mE\bigl[\, {n}^{-1} z_0(n) \bigr]  = \mE[ x_0(0)]
   = u(u+1)\bE \Bigl[\, \frac{p}{1+up} \, \Bigr]
              \end{align*}
and then convex analysis takes over.  We omit the remaining details of the
proof of \eqref{bernoulli6}.

To prove \eqref{bernoulli2},   note that
\[\begin{split}
g(x) = \sup_{u\geq 0} \{xu - f(u)\}
      &\geq \sup_{u\geq 0} \Bigl\{xu -  \frac{\bar{p}u(u+1)}{1+u\bar{p}}
     \Bigr\}\\[3pt]
     &= \begin{cases}
        \frac{1}{\bar{p}}(\sqrt{1-x} - \sqrt{1 - \bar{p}})^2 &
        \bar{p} \leq x \leq 1\\
        0  & 0 \leq x \leq \bar{p}.\\
\end{cases}
\end{split}
\]
We used Jensen's inequality and  concavity of  $p\mapsto\frac{p}{1+up}$. From this
\[g^{-1}(x) \leq
\begin{cases}
\bar{p} -\bar{p}x + 2\sqrt{\bar{p}(1-\bar{p})x} & 0 \leq x \leq
\frac{1 - \bar{p}}{\bar{p}}\\
1  &  x>\frac{1 - \bar{p}}{\bar{p}}\\
\end{cases}
\]
 and  \eqref{bernoulli2}  follows.  \end{proof}

\section{Proof of Proposition \ref{pr:limit}}   \label{sec:limit}
We comment briefly on the proof of Proposition \ref{pr:limit}.
Further details can be found in \cite{lin-thesis}.
The flow of arguments is standard.  First one takes
an integer point  $(x,y) \in \bZ^2_+$ and applies Liggett's
version of the subadditive ergodic theorem to the process
 $Z_{m,n} = - T((mx, my),(nx, ny))$, $0 \leq m <n$, to
prove that $\Psi(x,y)$ exists and is finite.
Then rational $(x,y)$ and real $(x,y)$ are handled by
approximations. Along the way regularity properties of
$\Psi$ are established and used: superadditivity, homogeneity, concavity
and continuity.

All this works easily for the Bernoulli case
because last-passage times are uniformly  bounded in terms
of path length.
Consequently we can  assume that  Proposition \ref{pr:limit} has been proved for the Bernoulli case.
 For the general case   we
 check that for integer points  $(x,y) \in \bZ^2_+$
  the moment hypotheses of the subadditive ergodic
theorem
\cite[p.~358]{durrett} follow from
our assumptions \eqref{momass}, \eqref{tailass1}, and \eqref{tailass2}:
\begin{align*}
\mE  Z^+_{0,1}  \leq  \mE  |T((0, 0),(x, y))|
\leq \mE  \sum_{0\leq i\le x, 0\leq j\le y}| X(i,j)|
= (x+1)(y+1) \mE  |X(0,0)|<\infty.
\end{align*}
Next:
\begin{align*}
\frac{1}{n}\mE  Z_{0,n}
&\ge - \frac{1}{n}\mE  \max_{\pi \in \Pi(nx, ny)} \sum_{z\in \pi} X(z)_+
= - \frac{1}{n}\mE  \max_{\pi \in \Pi(nx, ny)} \sum_{z\in \pi}
\int_0^{\infty}
I(X(z) > u)\,du\\
& \geq - \frac{1}{n}\mE  \int_0^{\infty} \max_{\pi \in \Pi(nx, ny)}
\sum_{z\in \pi} I(X(z) > u)\,du
 = -  \frac{1}{n}\int_0^{\infty} \mE  \max_{\pi \in \Pi(nx, ny)}
\sum_{z\in
\pi} I(X(z) > u)\,du\\
& \geq  -  \int_0^{\infty} \sup_n \frac{1}{n} \mE  \max_{\pi \in
\Pi(nx, ny)} \sum_{z\in \pi} I(X(z) > u)\,du
= - \int_0^{\infty}  \Psi_{Ber[1-F(u)]}(x,y)\,du\\
& \geq -(y+4\sqrt{xy})\int_0^{\infty} \sqrt{1-\bE  F_0(u)}\, du- x
\int_0^{\infty} \bigl(1-\underset{\bP}{\essinf} F_0(u)\bigr) du.
\end{align*}
  $I(A)$ is the indicator function of event $A$. $\Psi_{Ber[1-F(u)]}(x,y)$ is
the limiting time constant for the Bernoulli model where the weights
 have distributions $P(X(i,j)=1)=1-F_j(u)=1-P(X(i,j)=0)$.
On the last line above we used the Bernoulli estimate \eqref{bernoulli4}.
   By assumptions \eqref{tailass1} and \eqref{tailass2},
 $\mE
Z_{0,n} \geq n\gamma$ for a constant  $\gamma > -\infty$.
These estimates  justify the application of the subadditive ergodic
theorem.  We omit the remaining details and consider Proposition
\ref{pr:limit} proved.

\section{Proof of Theorem \ref{thm:main}}
\label{sec:thm:main}

For the first lemma,  let $\{F_j\}$ and $\{G_j\}$ be ergodic
sequences of distributions defined on a common probability space
under probability measure $\P$.  In a later step of the proof we need to
assume $\{F_j\}$ i.i.d.
 Assume that both processes $\{F_j\}$
and $\{G_j\}$  satisfy the assumptions made in Theorem \ref{thm:main}.
With some abuse of notation we label the time constants, means, and
even random weights associated to the processes $\{F_j\}$ and
$\{G_j\}$ with subscripts $F$ and $G$.  So for example
$\mu_F=\bE(\int x\,dF_0(x))$. The symbolic subscripts $F$ and $G$
should not be confused with the random distributions $F_j$ and $G_j$
assigned to the rows of the lattice.
We write $\Psi_{Ber([G(x) -F(x)]_+)}$ for the limit of a Bernoulli model
with  weight
   distributions $P(X(i,j)=1)=(G_j(x)-F_j(x))_+=1-P(X(i,j)=0)$ where $x$ is a fixed
   parameter.  An  analogous
convention will be used  for other Bernoulli models  along the way.

\begin{lem}\label{lem:diff}
Assume $\{F_j\}$ and $\{G_j\}$ satisfy \eqref{tailass1},
\eqref{tailass2}, \eqref{tailass3} and \eqref{tailass4}. Then for
$\alpha>0$,
\begin{equation}
\label{eqn:diff}
\begin{split}
&\abs{\Psi_F(\alpha,1) - \Psi_G(\alpha,1) - (\mu_F - \mu_G)}\\
&\qquad \leq 8 \sqrt{\alpha} \int_{-\infty}^{+\infty}
\Bigl(\bE\abs{G_0(x)-F_0(x)}\Bigr)^{1/2} dx
 +  \alpha \int_{-\infty}^{+\infty}
\underset{\bP}{\esssup} |F_0(x) - G_0(x)|\, dx.
\end{split}
\end{equation}
\end{lem}

\begin{proof}
The right-hand side of \eqref{eqn:diff} is finite under the
assumptions on  $\{F_j\}$ and $\{G_j\}$. Couple the $F_j$ and
$G_j$ distributed weights
 in a standard way. Let $\{u(z):z =
(i,j) \in \bZ_+^2\} $ be i.i.d.\
Uniform$(0,1)$ random variables.
Set  $X_F(z) = F_j^{-1}(u(z))$, where
$F_j^{-1}(u) = \sup\{x : F_j(x)< u\}$, and similarly  $X_G(z) =
G_j^{-1}(u(z))$. Write $\mE$ for expectation over the
entire probability space of distributions and weights.
\begin{align*}
&\quad\Psi_F(\alpha,1) - \Psi_G(\alpha,1) \\
& = \lim_{n \rightarrow \infty} \frac{1}{n}\mE\max_{\pi \in
\Pi(\lfloor \alpha n \rfloor, n )} \sum_{z \in \pi} X_F(z) - \lim_{n
\rightarrow \infty} \frac{1}{n} \mE\max_{\pi \in \Pi(\lfloor \alpha n
\rfloor,  n  )} \sum_{z \in \pi}
X_G(z) \\
&\leq \varlimsup_{n \rightarrow \infty} \frac{1}{n} \mE \max_{\pi \in
\Pi(\lfloor \alpha n \rfloor, n )} \sum_{z \in \pi}\bigl(
X_F(z)-X_G(z) \bigr)\\
&=\varlimsup_{n \rightarrow \infty}   \frac{1}{n} \mE \max_{\pi \in
\Pi(\lfloor\alpha n \rfloor, n )}
\sum_{z \in \pi} \int_{-\infty}^{+\infty}
\Bigl\{I \bigl(X_G(z)\le x< X_F(z)\bigr) - I\bigl(X_F(z)\le x< X_G(z)\bigr) \Bigr\} dx\\
&\leq \varlimsup_{n \rightarrow \infty}   \frac{1}{n} \mE
\int_{-\infty}^{+\infty} \max_{\pi \in
\Pi(\lfloor\alpha n \rfloor, n )} \sum_{z \in \pi}  \Bigl\{I \bigl(X_G(z)\le x< X_F(z)\bigr) - I\bigl(X_F(z)\le x< X_G(z)\bigr) \Bigr\} dx.
\end{align*}
We check that Fubini allows us
 to interchange the integral and the expectation.
Since $F$ and $G$ are interchangeable  it  is enough
to consider the first indicator function from above.
Let $a$ be an integer $\ge\alpha$.
\begin{align*}
&\quad\int_{-\infty}^{+\infty} \frac{1}{n}  \mE   \max_{\pi \in
\Pi(\lfloor\alpha n \rfloor, n )} \sum_{z \in \pi}
I \bigl(X_G(z)\le x< X_F(z)\bigr) dx\\
&\leq \int_{-\infty}^{+\infty}   \sup_n \frac{1}{n}\mE \max_{\pi
\in \Pi(an, n )} \sum_{z \in \pi}  I
\bigl(X_G(z)\le x< X_F(z)\bigr) \, dx\\
&= \int_{-\infty}^{+\infty}  \Psi_{Ber([G(x) -F(x)]_+)} (a,1)\, dx\\
&\leq  \int_{-\infty}^{+\infty} \Bigl( \bE \abs{G_0(x) -F_0(x)}  +
4\sqrt{a} \bigl( \bE
\abs{G_0(x) -F_0(x)} \bigr)^{1/2} + a\, \underset{\bP}{\esssup}\abs{G_0(x) - F_0(x)}
\Bigr)\, dx\\
& <\infty
\end{align*}
by estimate \eqref{bernoulli3} and
 the finiteness of the right-hand side of \eqref{eqn:diff}.
Continue from the limit above by applying Fubini. Then take the limit
inside the $dx$-integral by dominated convergence,  justified
by the $n$-uniformity in the bound above. Finally apply again the
Bernoulli  estimate \eqref{bernoulli3}.
\begin{align*}
&\quad \Psi_F(\alpha,1) - \Psi_G(\alpha,1) \\
 &\le\varlimsup_{n \rightarrow \infty}
\int_{-\infty}^{+\infty}   \frac{1}{n}
\mE \max_{\pi \in \Pi(\lfloor\alpha n \rfloor, n )} \sum_{z \in \pi}  \Bigl\{I \bigl(X_G(z)\le x< X_F(z)\bigr) - I\bigl(X_F(z)\le x< X_G(z)\bigr) \Bigr\}\, dx\\
&\le  \int_{-\infty}^{+\infty}  \lim_{n \rightarrow \infty}
\frac{1}{n}    \Bigr\{ \mE \max_{\pi \in
\Pi(\lfloor\alpha n \rfloor, n )} \sum_{z \in \pi} I\bigl(X_G(z)\le x< X_F(z)\bigr)\\
&\qquad\qquad\qquad
+\; \mE\max_{\pi \in \Pi(\lfloor \alpha n \rfloor, n )} \sum_{z \in
\pi}\bigl(1 - I\bigl(X_F(z)\le x< X_G(z)\bigr)\bigr) -
\sum_{z \in \pi}1 \Bigr\}\, dx\\
 &= \int_{-\infty}^{+\infty}
\bigl\{\Psi_{Ber([G(x)-F(x)]_+)}(\alpha,1) +
\Psi_{Ber(1-[F(x)-G(x)]_+)}(\alpha,1)
-(1+ \alpha) \bigr\}dx\\
&\leq\int_{-\infty}^{+\infty}
\biggl\{\bE\bigl(G_0(x)-F_0(x)\bigr)_+ +1-
\bE\bigl(F_0(x)-G_0(x)\bigr)_+\\
&\qquad+ 4 \sqrt{\alpha} \Bigl(\,\sqrt{\bE\bigl(G_0(x)-F_0(x)\bigr)_+}+ \sqrt{\bE\bigl(F_0(x)-G_0(x)\bigr)_+}\,  \Bigr) \\
&\qquad\qquad  +\alpha\,\Bigl(\, \underset{\bP}{\esssup} [G_0(x)-F_0(x)]_+
+ 1-\underset{\bP}{\essinf} [F_0(x)-G_0(x)]_+\Bigr)
  -(1 + \alpha)    \biggr\} dx\\
&\leq (\mu_F-\mu_G) +8\sqrt{\alpha} \int_{-\infty}^{+\infty}
\sqrt{\bE|F_0(x)-G_0(x)|} \,dx
+ \alpha\int_{-\infty}^{+\infty} \underset{\bP}{\esssup} |G_0(x)-F_0(x)|\,dx.
\end{align*}
Interchanging $F$ and $G$ gives the bound from the other direction
 and concludes the proof.
\end{proof}

For a while we make two convenient assumptions: that the weights are
 uniformly bounded, so for
a constant  $M<\infty$, \be  \bP\{ \text{$F_0(-M)=0$ and
$F_0(M)=1$}\} =1, \label{M-bd}\ee and that variances are uniformly
bounded away from zero, so for a constant $0<c_0<\infty$, \be  \bP\{
\sigma^2(F_0)\ge c_0\}=1. \label{var-bd}\ee
 Note that then
\be c_0\le \sigma^2(F_0) \leq M^2 \quad  \text{$\bP$-a.s}\label{var-2}\ee
 and the conditions assumed
for Theorem \ref{thm:main} are trivially satisfied by the uniform
boundedness.

Henceforth $r=r(\alpha)$ denotes a positive integer-valued
function such that $r(\alpha)
\nearrow \infty$ as $\alpha\searrow 0$.
Tile the lattice with $1\times r$ blocks
 $B_r(x,y) = \{(x,
ry+k): k=0,1,..., r-1\}$
 for $(x,y)\in \bZ_+^2$. A coarse-grained last-passage model
is defined by adding up the weights in each block:
\[  X_r(z) = \sum_{v\in B_r(z)} X(v). \]
The distribution of the new weight $X_r(i,j)$
 on row $j\in\bZ_+$ of the rescaled
lattice is the convolution
$  F_{r,\,j}  =  F_{rj}* F_{rj+1}*\dotsm *F_{rj+r-1}.  $

We repeat  Lemma 4.4 from
\cite{Martin2004} with a sketch of the argument.

\begin{lem}\label{lem:blocks}
 Let $\Psi_F(x,y)$
and $\Psi_{F_r}(x,y)$ be the last passage time functions
obtained by using  $F_j$ and $F_{r,j}$ as the distributions on the
$j$th row, respectively. If
$r\to\infty$ and
 $r\sqrt{\alpha} \rightarrow 0$ as
$\alpha \downarrow 0$, then $$\lim_{\alpha \downarrow 0}
\frac{1}{\sqrt{\alpha}} \bigl\lvert\Psi_F(\alpha,1) -
\frac{1}{r}\Psi_{F_r}(\alpha r,1)\bigr\rvert =0. $$
\end{lem}

\begin{proof}
 Given a path  $\pi \in \Pi(m, nr-1)$, consider all the blocks
that it intersects; this gives a path  $\tilde{\pi} \in
\Pi(m,n-1)$ in the rescaled lattice such that
 $\bigl\lvert (\cup_{z\in \tilde{\pi}}
B_r(z))\triangle \pi \bigr\rvert \leq mr.$
Then by \eqref{M-bd}
$$\big|\max_{\pi \in \Pi(m, nr)} \sum_{z\in \pi} X(z) - \max_{\tilde{\pi} \in \Pi(m, n)} \sum_{z\in \tilde{\pi}} X_r(z)\, \big| \leq mrM.$$
Take  $m = \lfloor \alpha nr\rfloor$, divide through
 by $nr$, and the conclusion  follows.
\end{proof}

 Let
$\mu_{r,y}$ and $V_{r,y} $ be the mean
and variance of $F_{r,y}$:
$$\mu_{r,y} = \sum_{i=0}^{r-1} \mu_{ry+i},\quad \textit{ and}\quad
V_{r,y} = \sum_{i=0}^{r-1} \sigma^2_{ry+i}. $$
Let $\Phi_{r,y}$ be the distribution function of the normal
$\cN(\mu_{r,y}, V_{r,y} )$ distribution,
 and $\wt\Phi_{r,y}$ the
distribution function of  $\cN(r\mu_F, V_{r,y} )$. The
 difference between $\Phi_{r,y}$ and
$\wt\Phi_{r,y}$ is that the latter has a
non-random mean. We shall also find it convenient
to use $\{X_j\}$ as a sequence of independent variables
with (random) distributions $X_j\sim F_j$. For the
next lemma we need to assume $\{F_j\}$ an
i.i.d.\ sequence under $\P$.

 As in \cite{Martin2004}, a key step in
the proof is the replacement of the rescaled  weights
with Gaussian weights, which is undertaken in the next lemma.

\begin{lem}\label{lem:N-replace} Assume  $\{F_j\}$
i.i.d.\ under $\P$.
If $r\to\infty$ and
 $r\sqrt{\alpha} \rightarrow 0$ as $\alpha \downarrow 0$, then
\be \lim_{\alpha \downarrow 0} \frac{1}{r \sqrt{\alpha}}
|\Psi_{F_r}(\alpha r,1)-\Psi_{\Phi_r}(\alpha r,1)|=0.
\label{N-replace}\ee
\end{lem}

\begin{proof}
According to Theorem 5.17 of \cite{Petrov},
 independent mean 0 random variables
 $X_1, X_2, X_3, \dotsc$  satisfy the estimate
\[ \Bigl\lvert P\Bigl\{ B_r^{-1/2}{\sum_{i=1}^r X_i}\le x\Bigr\}- \Phi(x)
\,\Bigr\rvert  \leq A\frac{\sum_{i=1}^r
E|X_i|^3}{B_r^{{3}/{2}}}(1+|x|)^{-3}, \quad x\in\bR,
\]
where  $B_r =
\sum_{i=1}^r \Var(X_i) $, $\Phi$ is the standard normal
 distribution function, and
$A$ is  a constant that is independent of the distribution functions
of $X_1, X_2, \dotsc , X_r$. Then,
\begin{equation}\label{eqn:5.4.1}\begin{aligned}
| F_{r,y}(x) - \Phi_{r,y}(x)|
&\leq A \frac{\sum_{i=0}^{r-1} E|X_{ry+i} -
\mu_{ry+i}|^3}{(\sum_{i=0}^{r-1}
\sigma_{ry+i}^2)^{{3}/{2}}}\bigl(1+V_{r,y}^{-1/2}|x -
\mu_{r,y}|\bigr)^{-3}\\
&\le \frac{C}{\sqrt r} \bigl(1+M^{-1}r^{-1/2}|x -
\mu_{r,y}|\,\bigr)^{-3}
\end{aligned}\end{equation}
where the second inequality used the assumptions
 $P(|X_i | \leq M) =1$ and $\sigma^2_i \geq
c_0^2 >0$.

Armed with \eqref{eqn:5.4.1} we now estimate the right-hand
side of \eqref{eqn:diff} for the processes
$\{F_{r,y}\}_{y\in\bZ_+}$ and $\{\Phi_{r,y}\}_{y\in\bZ_+}$
and with $\alpha$ replaced by $\alpha r$.

For the first term on the right in  \eqref{eqn:diff},
note this Schwarz trick:  for a probability density
$f$ on $\bR$ and a function $H\ge 0$,
\[
\int \sqrt H\,dx = \int f^{1/2} \sqrt{f^{-1}H}\,dx \le \biggl(\,
\int f^{-1} H\,dx\biggr)^{1/2}. \] For the calculation below take
$\delta>0$ and $f(x)= c_1(1+\abs{x-r\mu_F}^{1+\delta})^{-1}$ for the
right constant $c_1=c_1(\delta)$.   Factors that depend on  $M$ and
$\delta$ are subsumed in a constant $C$.  Then
\begin{align*}
&\sqrt{\alpha r}  \int_{-\infty}^{+\infty}
\Bigl(\bE\abs{F_{r,0}(x)-\Phi_{r,0}(x)}\Bigr)^{1/2} dx \\
&\le C\alpha^{1/2} r^{1/4}
\int_{-\infty}^{+\infty} \biggl\{\bE\Bigl[ \bigl(1+M^{-1}r^{-1/2}|x -
\mu_{r,0}|\,\bigr)^{-3} \Bigr] \biggr\}^{1/2}dx   \\
&\le C\alpha^{1/2} r^{1/4} \biggl\{
\bE \int_{-\infty}^{+\infty} \bigl( 1+\abs{x-r\mu_F}^{1+\delta}\bigr)
  \biggl(1+\frac{|x -
\mu_{r,0}|}{M\sqrt r}\,\biggr)^{-3}\, dx \biggr\}^{1/2}
\intertext{\text{by a change of  variables $x=\mu_{r,0}+yM\sqrt{r}$}}
&= C\alpha^{1/2} r^{1/2} \biggl\{
\bE \int_{-\infty}^{+\infty}
\frac{1+\abs{\mu_{r,0}-r\mu_F+yM\sqrt{r} }^{1+\delta}}
 {(1+\abs{y})^3} \, dy \biggr\}^{1/2} \\
&\le    C\alpha^{1/2} r^{1/2} \Bigl\{
\bE \abs{\mu_{r,0}-r\mu_F}^{1+\delta} + r^{(1+\delta)/2}\Bigr\}^{1/2}
\le   C\alpha^{1/2} r^{(3+\delta)/4}.
\end{align*}
In the last step we used $\bE \abs{\mu_{r,0}-r\mu_F}^{1+\delta}
\le Cr^{(1+\delta)/2}$ which follows because $\mu_{r,0}-r\mu_F$
is a sum of $r$ bounded mean zero i.i.d.\ random variables.

For the second term on the right in  \eqref{eqn:diff},
\begin{align*}
&\alpha r\int_{-\infty}^{+\infty} \underset{\bP}\esssup
\abs{ F_{r,0}(x) - \Phi_{r,0}(x)}\, dx
\leq C\alpha r^{1/2} \int_{-\infty}^{+\infty}
\underset{\bP}\esssup \biggl(1+\frac{|x -
\mu_{r,y}|}{M\sqrt r}\,\biggr)^{-3}   dx\\
&\qquad \leq  C\alpha r^{1/2} \biggl\{\int_{-\infty}^{-rM}
\biggl( 1+
\frac{-rM-x}{M\sqrt{r}}\biggr)^{-3} dx + \int_{-rM}^{rM}  dx +
\int_{rM}^{+\infty} \biggl(1+ \frac{x-rM}{M\sqrt{r}}\biggr)^{-3}
 dx\biggr\}\\
&\qquad \leq   C\alpha r^{3/2}.
\end{align*}
To summarize, with these estimates and \eqref{eqn:diff}
we have
\[   \frac{1}{r \sqrt{\alpha}}
|\Psi_{F_r}(\alpha r,1)-\Psi_{\Phi_r}(\alpha r,1)|
\le   \frac{C}{r \sqrt{\alpha}}
(\alpha^{1/2} r^{(3+\delta)/4} +  \alpha r^{3/2}). \]
If $\delta$ is fixed small enough,
 assumptions $r\to\infty$ and $r\sqrt\alpha\to 0$ make this vanish as
$\alpha\to 0$.\end{proof}

 The next lemma makes a further approximation
that puts us in the situation where
 all sites have  normal variables with the
same mean.

\begin{lem}
Let $\Psi_{\Phi_r}$ and $\Psi_{\wt\Phi_r}$ be defined
as before, and again $r\sqrt{\alpha} \rightarrow 0$ as $\alpha
\rightarrow 0$. Then
 $$\lim_{\alpha \downarrow 0}
\frac{1}{r\sqrt{\alpha}}|\Psi_{\Phi_r} (\alpha r, 1 ) -
\Psi_{\wt\Phi_r} (\alpha r, 1 )|=0.$$
\end{lem}

\begin{proof}
For  $z=(i,j) \in \bZ^2_+$, let $X^{(r)}(z)$ have distribution
$\Phi_{r,j}$ so that
$\wt{X}^{(r)}(z) = X^{(r)}(z) - \mu_{r,j} + r\mu_F$
 has distribution
$\wt\Phi_{r,j}$.  Now estimate:
\begin{align*}
\Psi_{\wt\Phi_r} (\alpha r, 1 )& = \lim_{n \rightarrow
\infty} \frac{1}{n} \max_{\pi \in \Pi(\lfloor \alpha nr \rfloor,  n
)}
\sum_{z \in \pi} \wt{X}^{(r)}(z)\\
& \leq \lim_{n \rightarrow \infty} \frac{1}{n} \max_{\pi \in
\Pi(\lfloor \alpha nr \rfloor,  n )} \sum_{z \in \pi} X^{(r)}(z) +
\lim_{n \rightarrow \infty} \frac{1}{n} \max_{\pi \in \Pi(\lfloor
\alpha nr \rfloor,  n )} \sum_{z\in \pi}\bigl( - \mu_{r,j} +
r\mu_F\bigr)\\
&\leq \Psi_{\Phi_r} (\alpha r, 1 ) + \lim_{n \rightarrow \infty}
\frac{1}{n} \sum_{j=1}^n\bigl( - \mu_{r,j} + r\mu_F\bigr)+
\lim_{n \rightarrow \infty} \frac{1}{n} 2Mr\cdot \lfloor \alpha nr\rfloor\\
&= \Psi_{\Phi_r} (\alpha r, 1 ) + 2M\alpha r^2.
\end{align*}
The opposite bound
$\Psi_{\wt\Phi_r} (\alpha r, 1 )\geq  \Psi_{\Phi_r} (\alpha r, 1 )  - 2M\alpha
r^2
$  comes similarly, and the lemma follows.
\end{proof}


Let us separate the mean by letting
 $\overline{\Phi}_{r,y}$ denote
 the  $N(0, \sum_{i=0}^{r-1} \sigma^2_{ry+i} )$
 distribution
function. Since the last-passage functions of the normal distributions
satisfy
$\Psi_{\wt\Phi_r} (\alpha r, 1 ) = r\mu_F(1+\alpha r)+\Psi_{\overline{\Phi}^{(r)}} (\alpha r, 1 )$,
  we can  summarize the effect of the last three lemmas as
follows.

\begin{lem}   Assume  $\{F_j\}$
i.i.d.\ under $\P$, and assume $r=r(\alpha)$ satisfies
 $r\to\infty$ and
 $r\sqrt{\alpha} \rightarrow 0$ as $\alpha \downarrow 0$.
 Under assumptions \eqref{M-bd} and \eqref{var-bd}
\begin{equation}\label{eqn:main1}
\lim_{\alpha \downarrow 0} \frac{1}{\sqrt{\alpha}}|\Psi_{F} (\alpha
, 1 ) - \mu_F- \frac{1}{r}\Psi_{\overline{\Phi}^{(r)}} (\alpha r, 1 )| =
0.
\end{equation}
\label{summ-lm}\end{lem}

In order to deduce a  limit from  \eqref{eqn:main1}   we utilize the
explicitly computable   case   of exponential distributions from
\cite{TimoKrug}.  We need to match up the random variances of the
exponentials with the variances $\sigma_j^2$  of the sequence
$\{F_j\}$.  Thus, given the i.i.d.\ sequence of quenched variances
$\sigma_j^2=\sigma^2(F_j)$ that we have worked with up to now under
condition \eqref{var-2},  let $\xi_j=1/\sigma_j$ and
$G_j(x)=1-e^{-\xi_jx}$   the rate $\xi_j$ exponential distribution.
Then $\{\xi_j\}_{j\in\bZ_+}$ is  an i.i.d.\ sequence of bounded
random variables  $0<c\leq \xi_j\leq b$ with distribution $m$.  We
can assume $c$ is the exact lower bound: $m[c,c+\e)>0$ for each
$\e>0$.  $G_j$ has mean and variance $\mu(G_j)=\xi_j^{-1}$ and
$\sigma^2(G_j)=\xi_j^{-2}=\sigma_j^2$.

Assumptions  \eqref{tailass1} and \eqref{tailass2} are easily
checked, and so the last-passage function   $\Psi_{G}$ is well-defined.
We would like to apply Lemma \ref{summ-lm} to this exponential
model, but obviously assumption \eqref{M-bd} is not satisfied.
To get around this difficulty we  do the following approximation
which leaves the quenched means and variances intact.  We learned
this trick from \cite{Martin2004}.

Let  $Y_j$ denote a $G_j$-distributed
 random variable.  For a fixed $\tau>0$, let $$m_j =
E(Y_j| Y_j >\tau ) \quad\text{and}\quad  w_j = E(Y_j^2| Y_j >\tau ).$$
The quantities
\[\ppb_j = \frac{(m_j-\tau )^2}{(m_j-\tau )^2 + w_j - m_j^2}\quad \textrm{and} \quad u_j = \frac{w_j -
\tau ^2}{ m_j -\tau }-\tau \]
satisfy the equations
\[(1-\ppb_j)\tau  + \ppb_j u_j = m_j \quad \textrm{and} \quad  (1-\ppb_j)\tau ^2 + \ppb_j u_j^2 = w_j.\]
Then  $0\leq \ppb_j \leq 1 $, $u_j \geq \tau $ and $ w_j
\geq \tau^2$.
Define distribution functions
\begin{equation}\label{deftildeF}
\wt{G}_j(x) =
\begin{cases}
 G_j(x) & 0\leq x < \tau \\
 1-\ppb_j[1-G_j(\tau )] & \tau \leq x < u_j\\
 1 &  x\geq u_j.
\end{cases}
\end{equation}
$\wt Y_j \sim \wt G_j$ satisfies  $EY_j = E \wt{Y}_j$ and $EY_j^2 = E
\wt{Y}_j^2$.
 Moreover, for any fixed  $\tau >0$, $$u_j= \frac{E(Y_j^2|
Y_j>\tau )-\tau^2}{E(Y_j| Y_j>\tau )-\tau }-\tau  = \frac{2}{\ppa_j} + \tau  \leq \frac{2}{c}
+ \tau ,$$ so the distributions  $\{\wt G_j\}$ are all supported
on  the   nonrandom bounded interval $[0, 2/c+\tau ]$.
Consequently Lemma \ref{summ-lm} applies to $\wt G$.
We can draw the same conclusion for $G$ once we have the next estimate:

\begin{lem}\label{lem:tobdd}
Given $\e>0$, we can select $\tau $ large enough and define
$\wt{G}_j$ as in \eqref{deftildeF} so that
$$\varlimsup_{\alpha \downarrow 0} \frac{1}{ \sqrt{\alpha}}
|\Psi_{G}(\alpha ,1)-\Psi_{\wt{G}}(\alpha ,1)|< \e.$$
\end{lem}

\begin{proof}
This comes from an application of Lemma \ref{lem:diff}.
$G_j=\wt G_j$ on $(-\infty, \tau )$ and
 $1 - \wt{G}_j \leq 1 - G_j$ on all of $\bR$.
The integrals on the right-hand side of \eqref{eqn:diff} are finite
and can be made arbitrarily small by choosing $\tau $ large.
\end{proof}

Currently we have shown that
\begin{equation}\label{eqn:main2}
\lim_{\alpha \downarrow 0} \frac{1}{\sqrt{\alpha}}|\Psi_{G} (\alpha
, 1 ) - \bE \sigma_{0}- \frac{1}{r}\Psi_{\overline{\Phi}_{r}} (\alpha r,
1 )| = 0.
\end{equation}
It remains to perform
 an explicit calculation on $\Psi_G(\alpha,1)$.  As before,
utilize the notation $\mu_G =\bE \xi_0^{-1}$ and  $\sigma_G^2
 = \bE\xi_0^{-2} $.

\begin{lem}\label{exponential} For   random exponential
distributions with rates bounded away from zero,
$$\Psi_{G} (\alpha
, 1 ) = \mu_G - 2 \sigma_G \sqrt{\alpha}+O(\alpha).$$
\end{lem}

\begin{proof}
Recall the definition of the limit shape $\Psi_G(\alpha, 1)$ from \eqref{PsiGg}.
From \eqref{ga-dual} one can
read that $tg(1/t)$ is nondecreasing in $t$.
Thus by  \eqref{PsiGg}
 $\Psi_G(\alpha, 1) = t = t(\alpha)$ such that $t
g({1}/{t}) = \alpha$.

Next we argue that when
$\alpha$ is close enough to $0$,
$g({1}/{t}) = - u_0/t + a(u_0)$
for some $0< u_0 < u^*$ with $a'(u_0) = {1}/{t}$.
Since $a(0)=0$ and $ a(u^*-) =c$, strict concavity gives for $0<u<u^*$
\begin{align*}
\Bigl\{\int_{[c,\infty)}\frac{\xi}{ (\xi - c)^2} m(d\xi)\Bigr\}^{-1} =
a'(u^*-) &< a'(u) =
\Bigl\{\int_{[c,\infty)} \frac{\xi}{ (\xi - a(u))^2} m(d\xi)\Bigr\}^{-1} \\
&< a'(0+) =  \Bigl\{\int_{[c,\infty)} { \xi}^{-1}  m(d\xi)\Bigr\}^{-1}=
\frac{1}{\mu_G}.
\end{align*}
On the other hand,
$  0<   \Psi_G(\alpha, 1) - \mu_G \leq
C\sqrt{\alpha} +  C\alpha $
where the second inequality comes from comparing $\{G_j\}$  in \eqref{eqn:diff}
with identically zero  weights.
Thus  when $\alpha$ is small enough, ${1}/{t}$ is in
the range of $a'$.  Consequently   there exists  $u_0\in(0,u^*)$ such that $a'(u_0) =
{1}/{t}$, or equivalently,
\begin{equation}
\label{eqn:t1} \int_{[c,\infty)} \frac{\xi}{ (\xi - a(u_0))^2} m(d\xi)
=t.
\end{equation} From the choice of $t$, $\alpha = t g({1}/{t}) =
t\bigl(- u_0/t +a(u_0)\bigr) = -u_0 + ta(u_0)$ and so
\begin{equation} \label{eqn:t2} \Psi_G(\alpha, 1) = t =
\frac{\alpha}{a(u_0)} + \frac{u_0}{ a(u_0)} = \frac{\alpha}{a(u_0)}
+ \int_{[c,\infty)} \frac{1}{\xi-a(u_0)} d m(\xi).
\end{equation}
Combining  \eqref{eqn:t1} and \eqref{eqn:t2}  gives
\begin{equation}
\label{eqn:alhpa/a}\alpha = a(u_0)^2 \int_{[c,\infty)} \frac{1}{(\xi-
a(u_0))^2} m(d\xi).  \end{equation} From this
 \[  a(u_0)^2\sigma^2_G= a(u_0)^2 \int_{[c,\infty)} \frac{1}{\xi^2} m(d\xi)
\leq \alpha. \] Hence we have $0\le a(u_0)\le \sqrt \alpha
/\sigma_G.$

When $\alpha$  and hence $a(u_0)$ is   small, \eqref{eqn:alhpa/a} and the last
bound on $a(u_0)$
yield
\begin{equation}\begin{split}
0\le \frac{\alpha}{a(u_0)^2} - \sigma^2_G &= \int_{[c,\infty)}
\bigl[\frac{1}{(\xi - a(u_0))^2} - \frac{1}{\xi^2}\bigr] m(d\xi)\\
 &= \int_{[c,\infty)} \frac{2\xi a(u_0)-
 a(u_0)^2}{\xi^2(\xi-a(u_0))^2}m(d\xi)\\
 &\le 2\int_{[c,\infty)} \frac{ a(u_0)}{\xi(c-a(u_0))^2}m(d\xi) = O(\sqrt \alpha).
\end{split}
\end{equation}
Consequently  \[ \frac{\sqrt \alpha}{a(u_0)}-\sigma_G
=\frac{\frac{\alpha}{a(u_0)^2} - \sigma^2_G}{\frac{\sqrt
\alpha}{a(u_0)}+\sigma_G}= O(\sqrt \alpha).\]

Now we put all the above together to prove the lemma.
 \begin{align*}
&\Psi_G(\alpha, 1) -\mu_G - 2\sigma_G\sqrt\alpha \\
&= \frac{\alpha}{a(u_0)} + \int_{[c,\infty)} \frac{1}{\xi-a(u_0)} m(d\xi)
-\mu_G - 2\sigma_G\sqrt\alpha\\
  &= \frac{\alpha}{a(u_0)} + \int_{[c,\infty)} \Bigl[ \,\frac{1}{\xi} + \frac{1}{\xi^2} a(u_0) +
  O\bigl(a(u_0)^2\bigr)\Bigr] m(d\xi)-\mu_G - 2\sigma_G\sqrt\alpha\\
  &= \sqrt{\alpha} \Bigl( \,\frac{\sqrt\alpha}{a(u_0)}-\sigma_G\Bigr)
  + \sigma_G a(u_0)\Bigl( \sigma_G - \frac{\sqrt \alpha}{a(u_0)}    \Bigr)
   +  \alpha \cdot O\Bigl(\frac{a(u_0)^2}{\alpha}\Bigr)\\
&=O(\alpha) \quad \text{ as $\alpha\downarrow 0$. }\qedhere
 \end{align*}
\end{proof}

Combining Lemma \ref{exponential} and \eqref{eqn:main2} gives
\[\lim_{\alpha \downarrow 0} \frac{1}{\sqrt{\alpha}}\bigl\lvert
\frac{1}{r}\Psi_{\overline{\Phi}_{r}} (\alpha r, 1 ) - 2\sigma_G
\sqrt{\alpha}\bigr\rvert=0.
\]
Substitute this back into  \eqref{eqn:main1} and recall  that $\sigma_F = \sigma_G$.
The conclusion we get is
 \be \lim_{\alpha \downarrow 0}
\frac{1}{\sqrt{\alpha}}|\Psi_{F} (\alpha , 1 ) - \mu_F- 2 \sigma_F
\sqrt{\alpha}| = 0. \label{aux30}\ee

We  have proved Theorem \ref{thm:main} under   assumptions
\eqref{M-bd} and \eqref{var-bd}. We now lift \eqref{var-bd}. For
$\e>0$, let $\{W(z)\}$ be i.i.d weights with distribution $H$
defined by $P(W(z)=\pm\e)=1/2$.  Let $\wt F_j=F_j*H$ be the
distribution of the weight $ \wt X(i,j)=X(i,j)+W(i,j)$. Let $\Psi_H$
and $\Psi_{\wt F}$ be the time constants of the last-passage models
with weights $\{W(z)\}$  and $\{\wt X(z)\}$, respectively. The
Bernoulli bound \eqref{bernoulli3} gives the estimate $
\Psi_{H}(x,y)\le 4\e\sqrt{xy}$. The corresponding last-passage times
satisfy
\[  T_{\wt F}(z) -T_H(z)\le T_F(z) \le   T_{\wt F}(z) +\hat{T}_H(z) \]
where $\hat{T}_H(z)$ uses the weights $-W(z)$. In the limit \be
\Psi_{\wt F}(\alpha, 1) -  4\e\sqrt{\alpha}
 \le \Psi_F(\alpha,1) \le
\Psi_{\wt F}(\alpha, 1) +  4\e\sqrt{\alpha}. \label{eqn:addH}\ee
Since $\sigma^2(\wt F_j)=\sigma^2(F_j) + \e^2$ while $\mu_{\wt
F}=\mu_F$, and $\e>0$ can be arbitrarily small,
  this estimate
suffices for limit \eqref{aux30}.

As the last item of the proof of Theorem \ref{thm:main} we remove
the uniform boundedness assumption \eqref{M-bd}.  Suppose $\{F_j\}$
 satisfy  the conditions required for
  Theorem \ref{thm:main}, but there is no common bounded support.
For a fixed $M>0$ define the truncated distributions
\[F_{j,M}(x) =
\begin{cases} 1 & x \geq M\\
F_j(x) & -M \leq x <M\\
0 & x< -M.  \end{cases}\] Let
$\mu_{M}$, $\sigma^2_{M}$ and  $\Psi_{F_{M}}(x,y)$  be
quantities associated to  $\{F_{j,M}\}$.

From \eqref{eqn:diff} and the conditions assumed in Theorem
\ref{thm:main},     \begin{align*}
 &\frac{1}{\sqrt{\alpha}}\abs{\Psi_F(\alpha,1) - \Psi_{F_{M}}(\alpha,1) - (\mu - \mu_M)}\\
& \leq   8 \int_{-\infty}^{+\infty}
\Bigl(\bE\abs{F_0(x)-F_{0,M}(x)}\Bigr)^{1/2} dx
 +
 \sqrt{\alpha} \int_{-\infty}^{+\infty}
\underset{\bP}{\esssup} |F_0(x) - F_{0,M}(x)|\, dx\\
&=8  \Bigl [ \int_{-\infty}^{-M}
\Bigl(\bE\abs{F_0(x)}\Bigr)^{1/2} dx + \int_{M}^{\infty}
\Bigl(\bE\abs{1-F_0(x)}\Bigr)^{1/2} dx\Bigr]\\
&\quad +  \sqrt{\alpha}
\Bigl[\int_{-\infty}^{-M} \underset{\bP}{\esssup} |F_0(x)|\, dx+
\int_{M}^{+\infty} \underset{\bP}{\esssup} |1-F_0(x) |\, dx\Bigr]  \le \e.
\end{align*}
The last inequality comes from choosing $M$ large enough, and is
valid for all   $\alpha \leq 1$.
Since $\bE (EX^2(0,0))   < \infty$, dominated convergence gives
  $\sigma_M \rightarrow \sigma$  and so
 we can pick $M$ so that
$\abs{\sigma - \sigma_M} <\e$.  Now
\begin{align*}
\frac{1}{\sqrt{\alpha}}\abs{\Psi_{F} (\alpha , 1 ) - \mu- 2 \sigma
\sqrt{\alpha}}
 \leq   \frac{1}{\sqrt{\alpha}}\abs{\Psi_{F_M}
(\alpha , 1 ) - \mu_M- 2 \sigma_{M} \sqrt{\alpha}}    + 2\e.
 \end{align*}
 Since $\e$ is arbitrary  and limit \eqref{aux30} holds for
$\{F_{j,M}\}$,  we  get the conclusion for the sequence $\{F_j\}$.
This concludes the proof of Theorem \ref{thm:main}.

\section{Proof of Theorem \ref{1,a-thm}}
\label{sec:1,a-thm}

\begin{proof}[Proof of Theorem \ref{1,a-thm}.]
   The lower bound  in \eqref{1,a-bd}  can be  proved by applying Martin's result
\eqref{martin-1}  to the homogeneous  problem where a maximal path is constructed
by using only those rows $j$  where $F_j=H_{i^*}$, the distribution with the
maximal mean $\mu^*=\mu(H_{i^*})$.  This is fairly straightforward and we leave the
details to the reader.

To prove the upper bound  in \eqref{1,a-bd}, we   start by increasing all the weights $X(z)$  by moving their
means to $\mu^*$.   Then we subtract the common mean $\mu^*$ from the  weights,
so that for the proof we can assume that all distributions $H_1,\dotsc, H_L$
have {\sl mean zero.}

Create the following coupling.  Independently of the process $\{F_j\}$,
let $\{X_\ell(z): 1\le \ell\le L, \, z\in\bZ_+^2\}$ be a collection of independent
weights such that  $X_\ell(z)$ has distribution  $H_\ell$.  Then define the weights
used for computing $\Psi(1,\alpha)$ by
\[   X(z)= \sum_{\ell=1}^L   X_\ell(z) I_{\{F_j=H_\ell\}}
\quad\text{for $z=(i,j)\in\bZ_+^2$.}  \]
Begin with  this elementary bound:
\be\begin{aligned}
\Psi(1,\alpha) &= \lim_{n \rightarrow \infty} \frac{1}{n}\mE\Bigl[ \;\max_{\pi \in
\Pi(n, \lfloor \alpha n \rfloor)} \sum_{z \in \pi} X(z)\,\Bigr] \\
&\le \sum_{\ell=1}^L   \lim_{n \rightarrow \infty} \frac{1}{n}\mE\Bigl[ \;\max_{\pi \in
\Pi(n, \lfloor \alpha n \rfloor)} \sum_{z \in \pi} X_\ell(z) I_{\{F_j=H_\ell\}}\,\Bigr].
\end{aligned}  \label{line-b5} \ee
The next lemma contains a convexity argument that will remove the indicators from the last
passage values above.

\begin{lem} Let $\cD$ be a sub-$\sigma$-field on a probability space $(\Omega, \cF, P)$,
$D$ an event in $\cD$,  and $\xi$ and $\eta$ two integrable random variables.  Assume that
$E\eta=0$, $\eta$ is independent of $\cD$, and   $\xi$ and $\eta$ are independent
conditionally on $\cD$.
Then  $E[\,\xi\vee (\eta I_D)\,]\le E[\,\xi\vee\eta\,]$.
\label{conv-lem}\end{lem}
\begin{proof}   By Jensen's inequality, for any fixed $x\in\bR$,
\[   x\vee E(\eta\,\vert\,\cD) \le  E( x\vee \eta\,\vert\,\cD). \]
Since $\eta$ is independent of $\cD$ and mean zero,
\[   x\vee 0 \le  E( x\vee \eta\,\vert\,\cD). \]
Integrate this against the conditional distribution $P(\xi\in dx\,\vert \,\cD)$ of $\xi$,
given $\cD$, and use the conditional independence of $\xi$ and $\eta$:
\[   E( \xi\vee 0 \,\vert\,\cD)  \le  E( \xi\vee \eta\,\vert\,\cD). \]
Next integrate this over the event $D^c$:
\[  E\bigl[ I_{D^c} \cdot\,\xi\vee (\eta I_D)\, \bigr]   =E\bigl[ I_{D^c} \cdot \, \xi\vee 0\,\bigr]
\le  E\bigl[ I_{D^c} \cdot\,\xi\vee \eta \, \bigr].
\]
The corresponding integral over the event $D$ needs no argument.
\end{proof}

Fix a lattice point $z_0=(i_0,j_0)$ for the moment.  We split the maximum in \eqref{line-b5}
according to whether the path $\pi$ goes through $z_0$ or not, and in case it goes
we also separate the weight at $z_0$:
\[ \max_{\pi
\in \Pi(n, \lfloor n\alpha  \rfloor )} \sum_{z \in \pi} X_\ell(z) I_{\{F_j=H_\ell\}}=
 A\vee  \bigl(B + X_\ell(z_0) I_{\{F_{j_0}=H_\ell\}} \bigr)
 = B +\bigl(A-B\bigr)\vee \bigl(X_\ell(z_0) I_{\{F_{j_0}=H_\ell\}} \bigr)
\] where
\[ A = \max_{\pi\not\ni z_0 } \sum_{z\in \pi} X_\ell(z)I_{\{F_j=H_\ell\}}
\quad \text{and}\quad
 B = \max_{\pi\ni z_0  } \sum_{z\in \pi\setminus\{z_0\} } X_\ell(z)I_{\{F_j=H_\ell\}} .\]
Now apply Lemma \ref{conv-lem} with $\xi=A-B$, $\eta=X_\ell(z_0)$, and
$D=\{F_{j_0}=H_\ell\}$.  Given $F_{j_0}$,  $A-B$ does not look at $X_\ell(z_0)$, so the
independence assumed in Lemma \ref{conv-lem}  is satisfied.  The outcome from that
lemma is the inequality
\begin{align*}
\mE\Bigl[ \;\max_{\pi \in
\Pi(n, \lfloor \alpha n \rfloor)} \sum_{z \in \pi} X_\ell(z) I_{\{F_j=H_\ell\}}\,\Bigr]
\le
\mE\bigl[  A\vee (B + X_\ell(z_0)) \bigr].
\end{align*}
This is tantamount to replacing the weight $X_\ell(z_0) I_{\{F_{j_0}=H_\ell\}} $ at $z_0$  with
$X_\ell(z_0)$.

 We can repeat this at all lattice points $z_0$ in  \eqref{line-b5}.  In the end we have
 an upper bound in terms of homogeneous last-passage values, to which we
 can apply Martin's result \eqref{martin-1}:
\begin{align*}
\Psi(1,\alpha) & \le \sum_{\ell=1}^L   \lim_{n \rightarrow \infty} \frac{1}{n}\mE\Bigl[ \;\max_{\pi \in
\Pi(n, \lfloor \alpha n \rfloor)} \sum_{z \in \pi} X_\ell(z)  \,\Bigr]= \sum_{\ell=1}^L \Psi_{H_\ell}(1,\alpha)
\\
&=  2\sqrt\alpha  \sum_{\ell=1}^L \sigma(H_\ell)  + o(\sqrt\alpha).
\end{align*}
This completes the proof of Theorem \ref{1,a-thm}.
\end{proof}

\section{Proofs for the exponential model}
\label{sec:exp-thm}

\begin{proof}[Proof of Theorem \ref{exp-thm}] Equation \eqref{PsiGg} gives
  \be \Psi_G(1,\alpha) =
\inf\{t \geq 0: tg({\alpha}/{t}) \geq 1\}  = t(\alpha) =  t . \ee
That the   infimum is achieved  can be seen from \eqref{ga-dual}.

Under Case 1 the critical value $u^* = \int_{[c,\infty)} c{ (\xi
-c)^{-1} } \,{ m(d\xi)}<\infty$, and also
\[  a'(u^*-)= \biggl\{ \int_{[c,\infty)}  \frac{\xi}{ (\xi - c)^2} \,m(d\xi)\biggr\}^{-1} > 0. \]
By the concavity of $a$ and \eqref{ga-dual},
for $0\le y\le a'(u^*-)$ we have  $g(y)=-yu^*+c$.
Consequently for small enough $\alpha$
\[   1 = tg({\alpha}/{t}) = -\alpha c \int_{[c,\infty)}  \frac{1}{\xi-c}
m(d\xi) + ct\] and equation \eqref{expcase1} follows.

In Case 2 $a'(0+)> a'(u^*-)= 0$ and hence for small enough $\alpha>0$  there exists a unique
$u_0\in(0,u^*) $ such that $a'(u_0) = {\alpha}/{t}$.  Set  $a_0 = a(u_0)\in(0,c)$.
As $\alpha\searrow 0$, both $u_0\nearrow u^*$ and $a_0\nearrow c$.
We have the
equations
 \[   a'(u_0)^{-1} =\int_{[c,\infty)}
\frac{\xi}{ (\xi - a_0)^2} \,m(d\xi) = \frac{t}{\alpha} \,, \quad  1
= t g({\alpha}/{t})  = -\alpha u_0 + ta_0\,,  \]
 \begin{equation}\label{case2}   \Psi_G(1,\alpha) \, =\, t \, = \, \frac{1}{a} + \frac{\alpha
u_0}{ a_0} \, = \, \frac{1}{a_0} + \alpha \int_{[c,\infty)}
\frac{1}{\xi-a_0} \,m(d\xi)\end{equation}and  \be
\label{excases}\frac{1}{a_0^2} = \alpha \int_{[c,\infty)}
\frac{1}{(\xi-a_0)^2}\,m(d\xi).\ee

Assuming $\eqref{exp-ass1}$,   start  with $\nu \in (-1,0) \cup
(0,1)$. For a small enough $\e>0$ there are constants
$0<\kappa_1<\kappa_2$ such that \be \kappa_1(\xi-c)^{\nu+1}\le
{m[c,\xi)}\le\kappa_2{(\xi-c)^{\nu+1}} \quad\text{ for $\xi
\in[c,c+\e]$ }\label{exp9}\ee and as $\e\searrow 0$ we can take
$\kappa_1,\kappa_2\to \kappa$. First   we estimate $c-a_0$.  Fix
$\e>0$.
\begin{align*}
\frac{1}{\alpha} &= a_0^2 \int_{[c,\infty)} \frac{1}{(\xi - a_0)^2}
\,m(d\xi)
= 2{a_0^2}\int_c^\infty  \frac{m[c,\xi)}{(\xi - a_0)^3}\,d\xi\\
&=2{a_0^2}\int_c^{c+\e}  \frac{m[c,\xi)}{(\xi - a_0)^3}\,d\xi +
C_1(\e)
\end{align*}
for a quantity $C_1(\e)=O(\e^{-2})$.  The first term above can be bounded
above and below by \eqref{exp9}, and we develop both bounds together
for $\kappa_i$, $i=1,2$,  as
\be\begin{aligned}
&2\kappa_i {a_0^2}\int_c^{c+\e}  \frac{(\xi-c)^{\nu+1}}{(\xi - a_0)^3}\,d\xi  +  C_1(\e)  \\
 &= 2\kappa_i a_0^2 \int_{c}^{c+\e} \frac{\bigl[(\xi-a_0)-(c-a_0)\bigr]^{\nu+1}}{(\xi-a_0)^3}
 d\xi
  +  C_1(\e) \\
 & =  2\kappa_i  a_0^2 \sum_{k=0}^\infty  \binom{\nu+1}{k}  (-1)^k(c - a_0)^k
 \int_{c}^{c+\e} (\xi-a_0)^{\nu - k-2} d\xi +  C_1(\e) \\
& = 2\kappa_i  a_0^2 \sum_{k=0}^\infty  \binom{\nu+1}{k}  (-1)^k(c - a_0)^k \frac{(c - a_0)^{\nu-k-1} - (c+\e-a_0)^{\nu-k-1} }{k-\nu+1} +  C_1(\e) \\
& = 2\kappa_i  a_0^2 A_\nu (c - a_0)^{\nu-1}
-  2\kappa_i  a_0^2 \sum_{k=0}^\infty  \binom{\nu+1}{k} \frac{(-1)^k}{k-\nu+1}
(c - a_0)^k (c+\e-a_0)^{\nu-k-1}   +  C_1(\e) \\
& = 2\kappa_i  a_0^2 A_\nu (c - a_0)^{\nu-1} +  C_1(\e).
\end{aligned} \label{1overalpha}\ee
$C_1(\e)$ changed of course in the last equality.
In the next to last equality above we defined
\[   A_\nu= \sum_{k=0}^\infty  \binom{\nu+1}{k} \frac{(-1)^k}{k-\nu+1}. \]
 Rewrite the above development in the form
\[   (c - a_0)^{1-\nu} =  2\kappa c^2 A_\nu\alpha  +
\alpha[ 2A_\nu(\kappa_i  a_0^2-\kappa c^2)+C_1(\e)(c - a_0)^{1-\nu}].  \]
Now choose $\e=\e(\alpha)\searrow 0$ as $\alpha\searrow 0$ but slowly enough
so that $C_1(\e)(c - a_0)^{1-\nu}\to 0$ as $\alpha\searrow 0$.  Then
also $\kappa_i  a_0^2\to\kappa c^2 $ and
we can write
\be c-a_0 = B_0\alpha^{\frac1{1-\nu}} +o(\alpha^{\frac1{1-\nu}})  \label{exp12}\ee
with a new constant $B_0=(2\kappa c^2 A_\nu)^{\frac1{1-\nu}}$.

Now consider the case   $\nu \in (0,1)$ which also guarantees
$\int_{[c,\infty)} (\xi - c)^{-1} \,m(d\xi)<\infty$.  From
\eqref{case2} and \eqref{exp12} as  $\alpha\searrow 0$
 \be \begin{aligned}&\Psi_G(1,\alpha)  = \frac{1}{a_0}+ \alpha \int_{[c,\infty)} \frac{1}{\xi - a_0}\, m(d\xi)\\
&=\frac{1}{c}+ \alpha \int_{[c,\infty)} \frac{1}{\xi - c} m(d\xi)
+O(\alpha^{\frac1{1-\nu}})  + \alpha  \biggl( \,  \int_{[c,\infty)}
\frac{1}{\xi - a_0}\, m(d\xi)
-   \int_{[c,\infty)} \frac{1}{\xi - c} \,m(d\xi) \biggr) \\
&=\frac{1}{c}+ \alpha \int_{[c,\infty)} \frac{1}{\xi - c} \,m(d\xi)
+o(\alpha).
 \nn\end{aligned}\ee

Next the case $\nu \in (-1,0)$.  The steps are similar to those above
 so we can afford to be
sketchy.
 \begin{align*}
&\Psi_G(1,\alpha)= \frac{1}{a_0} +  \alpha \int_{[c,\infty)}
\frac{1}{\xi-a_0}
\,m(d\xi)\\
&=\frac1c + \frac{c-a_0}{c^2}   + \frac{(c-a_0)^2}{c^2a_0} +\alpha
\int_c^{c+\e}\frac{m[c,\xi)}{(\xi-a_0)^2}\,d\xi + \alpha  C_1(\e).
\end{align*}
Again, using \eqref{exp9} and proceeding as in \eqref{1overalpha},
we develop an upper and a lower bound for the quantity above  with
distinct constants $\kappa_i$, $i=1,2$. After bounding $m[c,\xi)$
above and below with $\kappa_i(\xi-c)^{\nu+1}$  in the integral,
write $(\xi-c)^{\nu+1}=((\xi-a_0)-(c-a_0))^{\nu+1}$ and expand in
power series.
 \begin{align*}
&\frac1c + B_0c^{-2}\alpha^{\frac1{1-\nu}} +o(\alpha^{\frac1{1-\nu}})
+\alpha \kappa_i \int_c^{c+\e}\frac{(\xi-c)^{\nu+1}}{(\xi-a_0)^2}\,d\xi  + \alpha  C_1(\e) \\
&=\frac1c + B_0c^{-2}\alpha^{\frac1{1-\nu}} +o(\alpha^{\frac1{1-\nu}})
+\alpha \kappa_i  (c-a_0)^\nu \sum_{k=0}^\infty  \binom{\nu+1}{k}  \frac{(-1)^k}{k-\nu} \\
&\qquad \qquad \qquad
+ \alpha \kappa_i  (c-a_0+\e)^\nu \sum_{k=0}^\infty  \binom{\nu+1}{k}  \frac{(-1)^k}{\nu-k}
\biggl(\frac{c-a_0}{c-a_0+\e}\biggr)^k
+ \alpha  C_1(\e) \\
&= \frac1c + B\alpha^{\frac1{1-\nu}} +o(\alpha^{\frac1{1-\nu}})
+  A_{\nu, 2}  \alpha (\kappa_i-\kappa)   (c-a_0)^\nu   + \alpha  C_1(\e).
 \end{align*}
In the last equality the next to last term with the  $\sum_{k=0}^\infty $ sum was subsumed in the $ \alpha  C_1(\e)$
term.  Then
 we introduced   new constants
\be A_{\nu, 2} = \sum_{k=0}^\infty  \binom{\nu+1}{k}  \frac{(-1)^k}{k-\nu}
\quad\text{and}\quad
  B= B_0c^{-2} + \kappa B_0^\nu A_{\nu, 2}. \label{exp18}\ee
  As before, by letting $\e=\e(\alpha)\searrow 0$  slowly enough as $\alpha\searrow 0$
we can extract  $\Psi_G(1,\alpha)=c^{-1} + B\alpha^{\frac1{1-\nu}} +o(\alpha^{\frac1{1-\nu}})$
from the above bounds.

It remains to treat the cases $\nu=-1, 0, 1$   where integration of the type
done in \eqref{1overalpha} is elementary.
We omit the details.  \end{proof}


\begin{thebibliography}{10}

\bibitem{andj-etal}
E.~D. Andjel, P.~A. Ferrari, H.~Guiol, and C.~Landim.
\newblock Convergence to the maximal invariant measure for a zero-range process
  with random rates.
\newblock {\em Stochastic Process. Appl.}, 90(1):67--81, 2000.

\bibitem{baik-suid}
Jinho Baik and Toufic~M. Suidan.
\newblock A {GUE} central limit theorem and universality of directed first and
  last passage site percolation.
\newblock {\em Int. Math. Res. Not.}, (6):325--337, 2005.

\bibitem{bodi-mart}
Thierry Bodineau and James Martin.
\newblock A universality property for last-passage percolation paths close to
  the axis.
\newblock {\em Electron. Comm. Probab.}, 10:105--112 (electronic), 2005.

\bibitem{durrett}
Rick Durrett.
\newblock {\em Probability: Thoery and Examples}.
\newblock Duxbury Press, 2004.

\bibitem{georgiou-10}
Nicos Georgiou.
\newblock Soft edge results for longest increasing paths on the planar lattice.
\newblock {\em Electron. Commun. Probab.}, 15:1--13, 2010.

\bibitem{glyn-whit}
Peter~W. Glynn and Ward Whitt.
\newblock Departures from many queues in series.
\newblock {\em Ann. Appl. Probab.}, 1(4):546--572, 1991.

\bibitem{grav-trac-wido-01}
Janko Gravner, Craig~A. Tracy, and Harold Widom.
\newblock Limit theorems for height fluctuations in a class of discrete space
  and time growth models.
\newblock {\em J. Statist. Phys.}, 102(5-6):1085--1132, 2001.

\bibitem{grav-trac-wido-02b}
Janko Gravner, Craig~A. Tracy, and Harold Widom.
\newblock Fluctuations in the composite regime of a disordered growth model.
\newblock {\em Comm. Math. Phys.}, 229(3):433--458, 2002.

\bibitem{GTW2002}
Janko Gravner, Craig~A. Tracy, and Harold Widom.
\newblock A growth model in a random environment.
\newblock {\em The Annals of Probability}, 30(3):1340--1368, 2002.

\bibitem{joha}
Kurt Johansson.
\newblock Shape fluctuations and random matrices.
\newblock {\em Comm. Math. Phys.}, 209(2):437--476, 2000.

\bibitem{joha01}
Kurt Johansson.
\newblock Discrete orthogonal polynomial ensembles and the {P}lancherel
  measure.
\newblock {\em Ann. of Math. (2)}, 153(1):259--296, 2001.

\bibitem{krug-ferr}
Joachim Krug and Pablo Ferrari.
\newblock Phase transitions in driven diffusive systems with random rates.
\newblock {\em J. Phys. A}, 29:L465--L471, 1996.

\bibitem{lin-thesis}
Hao Lin.
\newblock Properties of the limit shape for some last
passage growth models in random
environments.
\newblock {\em {\rm UW-Madison doctoral thesis, 2011}}.

\bibitem{Martin2004}
James Martin.
\newblock Limiting shape for directed percolation models.
\newblock {\em The Annals of Probability}, 32(4):2908--2937, 2004.

\bibitem{Petrov}
V.~V. Petrov.
\newblock {\em Limit Theorems of Probability Theory}.
\newblock Oxford Univ. Press, 1995.

\bibitem{rock-ca}
R.~Tyrrell Rockafellar.
\newblock {\em Convex analysis}.
\newblock Princeton Mathematical Series, No. 28. Princeton University Press,
  Princeton, N.J., 1970.

\bibitem{Rost1981}
H.~Rost.
\newblock Nonequilibrium behaviour of a many particle process: Density profile
  and local equilibria.
\newblock {\em Z. Wahrsch. Verw. Gebiete}, 58(1):41--53, 1981.

\bibitem{sepp97incr}
Timo Sepp{\"a}l{\"a}inen.
\newblock Increasing sequences of independent points on the planar lattice.
\newblock {\em Ann. Appl. Probab.}, 7(4):886--898, 1997.

\bibitem{Timo1998}
Timo Sepp\"{a}l\"{a}inen.
\newblock Exact limiting shape for a simplified model of first-passage
  percolation on the plane.
\newblock {\em The Annals of Probability}, 26(3):1232--1250, 1998.

\bibitem{TimoKrug}
Timo Sepp\"{a}l\"{a}inen and Joachim Krug.
\newblock Hydrodynamics and platoon formation for a totally asymmetric
  exclusion model with particlewise disorder.
\newblock {\em Journal of Statistical Physics}, 95(3-4):525--567, 1999.

\end{thebibliography}
\end{document}